\documentclass[12pt]{amsart}
\usepackage{color}

\usepackage[margin=1.0in]{geometry}

\usepackage{amsmath}
\usepackage{amsthm}
\usepackage{amssymb}
\usepackage[matrix,arrow,curve]{xy}
\usepackage{enumitem} 
\usepackage{graphicx}
\usepackage{mdwlist}	
\usepackage{mathrsfs}
\usepackage{tikz}
\usetikzlibrary{patterns}
\usepackage{caption}
\usepackage{subcaption}

\newcommand\setItemnumber[1]{\setcounter{enum\romannumeral\@enumdepth}{\numexpr#1-1\relax}}

 \newcommand{\p}{\mathbb{P}}

\newcommand{\resp}{\widetilde{\p}}
\newcommand{\resx}{\widetilde{X}}
\newcommand{\resb}{\widetilde{B}}

\newcommand{\respi}{\widetilde{\pi}}

\newcommand{\codes}{\overline{C}_B}

\DeclareMathOperator{\spl}{\mathit{sl}}

\DeclareMathOperator{\Pic}{Pic}
\DeclareMathOperator{\Sing}{Sing}
\DeclareMathOperator{\Hom}{Hom}
\DeclareMathOperator{\Ker}{Ker}
\DeclareMathOperator{\ev}{ev}
\DeclareMathOperator{\Coker}{Coker}

\DeclareMathOperator{\cork}{corank}
\DeclareMathOperator{\Syz}{Syz}

\def\mathhyphen{{\hbox{-}}}

\DeclareMathOperator{\F}{\mathcal{F}}
\DeclareMathOperator{\G}{\mathcal{G}}

\DeclareMathOperator{\E}{\mathcal{E}}
\DeclareMathOperator{\C}{\mathscr{C}}
\DeclareMathOperator{\I}{\mathcal{I}}

\DeclareMathOperator{\OP}{\mathcal{O}}

\DeclareMathOperator{\Xw}{S_w}

\DeclareMathOperator{\grmod}{\mathhyphen grmod}
\DeclareMathOperator{\dbcoh}{\mathcal{D}^b(\mathrm{Coh}(\p^3))}

\newtheorem{theorem}[equation]{Theorem}
\newtheorem{lemma}[equation]{Lemma}
\newtheorem{proposition}[equation]{Proposition}
\newtheorem{corollary}[equation]{Corollary}

\theoremstyle{definition}
\newtheorem{defin}[equation]{Definition}

\newtheorem{notation}[equation]{Notation}

\theoremstyle{remark}
\newtheorem{remark}[equation]{Remark}

\makeatletter
\@addtoreset{equation}{section}
\makeatother

\title{Non-rational sextic double solids}
\author{Alexandra Kuznetsova}
\address{National Research University Higher School of Economics, Russian Federation,
  Department of Mathematics, 6 Usacheva st., 119048 Moscow;
  \'Ecole Polytechnique, France, CMLS, Route de Saclay, 91128 Palaiseau.
}
\email{sasha.kuznetsova.57@gmail.com}
\begin{document}

\begin{abstract}
 We study double solids branched along nodal sextic surfaces in a projective space 
 and the \mbox{$2$-tor}sion subgroups in the  third integer
 cohomology groups of their resolutions of singularities. These groups 
 can be considered as obstructions to rationality of~the double solids. Studying
 these groups we conclude that all sextic double solids admitting non-trivial 
 obstructions to rationality are branched along determinantal surfaces of very specific type and
 we provide an explicit list of them. 
 \end{abstract}
 
 \maketitle

\section{Introduction}
In this work we  study  rationally connected varieties over the field of complex numbers that have some obstructions 
to rationality. As a result we get descriptions 
of non-rational varities which are however somewhat similar to rational. 

Rationally connected
curves or surfaces are necessarily rational: this is an easy corollary of the Hurwitz theorem and the Castelnuovo criterion.
For threefolds this is false; one of the first counterexamples was introduced by M.\,Artin and D.\,Mumford~\mbox{\cite{AM}}.
They considered so called \emph{quartic double solids}; namely, double covers of the projective space~$\p^3 = \p(V)$, branched
along a quartic. They constructed special nodal quartic surfaces $B$ and showed that the double covers branched along such $B$ are 
rationally connected, but non-rational. These surfaces are in the class of \emph{quartic symmetroids}; i.e., they are zero
loci of the determinants of $4\times 4$ matrices of linear forms over~$\p^3$. The general surface of this class
is nodal with 10 nodes.
In order to show non-rationality of a double solid, Artin and Mumford introduced a birational invariant of a smooth
projective variety $M$; namely the torsion subgroup~$T$ in~$H^3(M,\mathbb{Z})$. In particular, they showed that if the
group $T\ne0$, then $M$ is non-rational; moreover, it is not stably rational.

In view of this in~\mbox{\cite{AM}} Artin and Mumford proved that for a nodal quartic symmetroid~$B$ the cohomology group~\mbox{$H^3(\resx,\mathbb{Z})$} 
for any resolution $\resx$ of the double solid $X$ branched along quartic nodal surface $B$ contains a non-trivial 2-torsion subgroup~$T_2(\resx)$, therefore~$X$ is not rational.
Thus, the group~$T_2(\resx)$ provides an obstruction to rationality of the double solid.

Note that in the case of a double cover $M$ of a threefold with torsion-free integer cohomology groups
all $p$-torsion subgroups in $H^3(M,\mathbb{Z})$ are trivial for all prime $p>2$.

Afterwards, S.\,Endrass~\mbox{\cite{Endrass99}} studied all nodal quartic double solids and proved that the family of
quartics constructed by Artin and Mumford is the only family of nodal quartics such that double solids branched along them
have non-trivial group $T_2(\resx)$. 
% In particular, any quartic in this family has exactly 10 nodes. 

Our goal is to study the next interesting 
class of rationally connected varieties: double solids branched along  
nodal sextic surfaces. Note that this is the last case when the question about rationality of a nodal double solid is not obvious. 
Indeed, nodal double solids branched along surfaces of degree greater than 6 have  non-negative Kodaira dimensions;
therefore, they are non-rational.

The case of sextic double solids contains more complicated examples of non-rational varieties than the case of quartic 
double solids. There are known examples~\mbox{\cite[Section 3.3]{ex-IKP}} and~\mbox{\cite{ex-Beau}} of such varieties with non-trivial group $T_2(\resx)$ with 
different number of nodes. We are going to 
provide an explicit classification of obstructed sextic double solids:
we show that if a nodal sextic double solid admits a non-trivial groups~$T_2(\resx)$, then it belongs
to one of four explicitly described families $\mathcal{Z}_{31}$, $\mathcal{Z}_{32}$, $\mathcal{Z}_{35}$ and $\mathcal{Z}_{40}$ (see Theorem~\ref{Result}). 
Moreover, a general element of each of these families actually admits a non-trivial obstruction, see Remark \ref{rmk_2}. 

Now let us give more details in order to describe our result. We will use the following notations.
\begin{notation}\label{notation_main}
 Let $B\subset \p^3 = \p(V)$ be a nodal surface of even degree. Denote by $\resp^3$ the blow up of $\p^3$ in all the singular points of $B$, by $\resb$
the proper transform of $B$ and by $\resx$ the double cover of $\resp^3$ branched along~$\resb$. 
Then we have the following diagram:
 \begin{equation*}
  \xymatrix{\resx \ar[d]^{\sigma_X} \ar[r]^{\respi}_{2:1} & \resp^3 \ar[d]^{\sigma} & \resb \ar[d]^{\sigma_B} \ar@{_{(}->}[l]_{\widetilde{i}} \\
           X \ar[r]^{\pi}_{2:1} & \p^3 & B \ar@{_{(}->}[l]_{i}  
           }
 \end{equation*}
\end{notation}
We are going to study 
the group~$T_2(\resx)$ using the methods introduced by Endrass. He established a connection between 2-torsion 
subgroups~\mbox{$T_2(\resx)\subset H^3(\resx,\mathbb{Z})$} and so-called even sets of nodes of the surface~$B$. 
\begin{defin}\label{def_esn}
A set 
$w\subset\mathrm{Sing}(B)$ on a nodal surface $B\subset\p^3$ is called \emph{a~\mbox{$\delta/2$-ev}en set of nodes} for $\delta = 0$ or 1, 
if the following divisor on $B$ is an even element of the Picard group.
\begin{equation*}
 E_w = \left.\left(\delta H+\sum_{p\in w} E_p\right)\right|_{\resb}\in 2\cdot \Pic(\resb).
%  \ \left( \text{respectively } H+\sum_{p\in w} E_p|_{\resb} \divby 2 \text{ in }\Pic(\resb)\right), 
\end{equation*}
Here by $E_p$ and $H$ we denote the exceptional divisors of the blow up $\resp^3$ and the inverse image of the class of a plane on~$\p^3$ respectively.
\end{defin}
For both $\delta=0$ and 1 such sets are called \emph{even sets of nodes}.
These sets and their properties were studied by F.\,Catanese~\mbox{\cite{Catanese}}, W.\,Barth~\mbox{\cite{Barth}}, 
A.\,Beauville~\mbox{\cite{Beauville}}, and D.\,B.\,Jaffe and D.\,Ruberman~\mbox{\cite{Jaffe-Ruberman}}. 
Their main property is that even sets of nodes form a vector space over the field~$\mathbb{F}_2$ of two elements,
with respect to the operation of symmetric difference; we will denote the space of all even sets of nodes on the surface $B$ by $\codes$. 
An important assertion, proved by Endrass~\mbox{\cite[Lemma 1.2]{Endrass99}}, claims that the group $T_2(\resx)$ is a $\mathbb{F}_2$-subspace of $\codes$,
so that there is a non-canonical direct sum decomposition:
\begin{equation}\label{decomp-codes}
 \codes 
%  = \mathrm{Coker}\left({\respi^*}:H^2(\resp^3, \mathbb{F}_2)  \hookrightarrow H^2(\resx, \mathbb{F}_2)\right)
\cong T_2(\resx) \oplus \mathbb{F}_2^d.
\end{equation}
Here the number $d$ is called \emph{the defect of the set $\Sing(B)$} and it has a geometric interpretation, for more details see Section~\mbox{\ref{sect_defect}}.
% \begin{defin}
% For any set of points~$\{p_1,\dots,p_k\}$ on $\p^3$ and any integer $N$ we define an \emph{$N$-defect of this set}:
%   \begin{equation*}\label{def_of_defect}
%    d_N(\{p_1,\dots,p_k\})= \dim \left(H^0(\p^3,\OP(N-p_1-\dots-p_k))\right) - (\dim\left(H^0(\p^3,\OP(N))\right) - k).
%   \end{equation*} 
% \end{defin}
% And by \cite[\S 3]{Clemens} in \eqref{decomp-codes} we have $d= d_{3 \deg(B)/2-4}(\Sing(B))$ is a $5$-defect
% (we will write just defect for simplicity) of $\mathrm{Sing}(B)$.

Another important fact about even sets of nodes is that any nodal surface containing such a set of nodes is
determinantal:
\begin{theorem}[{\cite{CasCat}, \cite[Lemma 4]{Barth}, cf. \cite[Corollary 0.4]{CasCat_with_error}}]
\label{Catanese_Casnati}
  If a nodal surface~$B$ in $\p^3$ contains a non-empty $\delta/2$-even set of nodes $w$, 
%   and $\G$ is a sheaf as in \eqref{def_of_G}, then there exist an exact sequence:
  then there exists a vector bundle $\E$ over~$\p^3$ and an injective morphism
  \begin{equation}\label{Phi}
   \Phi\colon\E^{\vee}(-\deg(B)-\delta) \hookrightarrow \E,
  \end{equation}
  such that $\Phi$ is symmetric, i.\,e. induced by an element of $H^0(\p^3,S^2\E(\deg(B)+\delta))$, and one has the following isomorphisms of schemes:
 \begin{equation*}
   \begin{split}
    &B = B(\E)=\{x\in\p^3\ |\ \mathrm{corank}(\Phi|_{x})\geqslant 1 \};\\
    &w = w(\E)= \{x\in\p^3\ |\ \mathrm{corank}(\Phi|_{x})=2\}.\\
   \end{split}
  \end{equation*}
  Since $B$ is nodal, the $0$-dimensional scheme $w$ can be considered as a set of points in $\p^3$.
  If the bundle $S^2\E(\deg(B)+\delta)$ is globally generated, then for a sufficiently general $\Phi$, we have the
  equality of sets~\mbox{$\Sing(B) = w$}.
\end{theorem}
\begin{defin}\label{def_CC}
 If $B$ is a nodal surface and $w$ is an even set of nodes on it, we call the bundle $\E$ from Theorem \ref{Catanese_Casnati} a 
 \emph{Casnati--Catanese bundle} of $w$. 
\end{defin}
In fact, the vector bundle $\E$ in Theorem \ref{Catanese_Casnati} is constructed explicitly. In the case of 0-even sets this leads 
to the following classification.
\begin{theorem}[{\cite[Main Theorem A, Proposition 3.3]{CatTon}}] \label{0-even_sets}
 If $w$ is a $0$-even set of nodes on a nodal sextic surface $B$, then its Casnati--Catanese bundle is one of the following:
 \begin{align*}
  &(a)\hspace{15pt} \E = \OP_{\p^3}(-1)\oplus \OP_{\p^3}(-2) &\text{ and $|w|=24$};\\
  &(b)\hspace{15pt} \E = \OP_{\p^3}(-2)^{\oplus 3} &\text{ and $|w|=32$};\\
  &(c)\hspace{15pt} \E = \Omega_{\p^3}^1(-1)\oplus \OP_{\p^3}(-2) &\text{ and $|w|=40$};\\
  &(d)\hspace{15pt} \E = \mathrm{Ker}\left(\Omega_{\p^3}^1(-1)^{\oplus 3} \to \OP_{\p^3}(-2)^{\oplus 3}\right)&\text{ and $|w|=56$}.\\
 \end{align*}
\end{theorem}
\begin{remark}\label{rmk_1}
 Let $B$ be a surface with a $0$-even set of 56 nodes $w$ whose Casnati--Catanese bundle is
 of the type $(d)$. Then there exists its $0$-even subset of 24 nodes $w'\subset w$ with
 Casnati--Catanese bundle of the type $(a)$. Similarly, the 0-even set of 32 nodes $w'' = w\setminus w'$ has 
 Casnati--Catanese bundle of the type $(b)$. For more details see Section \ref{sect_conclusion}.
\end{remark}

Our goal is to
study the case of $1/2$-even sets of nodes. We introduce an auxiliary notion of minimal even sets of nodes which allows us to
avoid the situations described in Remark~\ref{rmk_1}. 
\begin{defin}\label{def_minimal}
 We call~\mbox{$w\in \codes$} \emph{minimal} if it does not contain a proper even subset of nodes.
\end{defin}
This notion also happens to be very useful because it provides some cohomological restrictions for Casnati--Catanese bundles (see Section \ref{sect_minimal_1/2_even}). 
Furthermore, on each surface~$B$ with non-trivial group 
$\codes$ we can find a minimal non-empty even set of nodes. 
% Moreover, for each minimal non-empty even set of nodes on a nodal surface $B$ there exists
So using the notion of minimality we get a precise description 
of arbitrary surfaces which have non-trivial group $\codes$; this is the first result of this paper.
\begin{theorem}\label{if_codes_ne_0}
 Assume $B$ is a nodal sextic surface and $\codes\ne 0$. Then there exists an even set of nodes $w$ on $B$ such that either 
 $w$ is $0$-even and its Casnati--Catanese bundle is of type~$(a)$, $(b)$ or $(c)$
 or $w$ is $1/2$-even  and its Casnati--Catanese bundle is 
 \begin{equation}\label{eq_1/2-even}
  \E  = {\Omega_{\p^3}^1}(-2) ^{\oplus k} \oplus \bigoplus_{i\in\mathbb{Z}} \OP_{\p^3}(-i)^{\oplus m_i}
 \end{equation}
 for some non-negative integers $k$ and $m_i$.
\end{theorem}

Any  Casnati--Catanese bundle of a $1/2$-even set of nodes on a sextic has the form \eqref{eq_1/2-even},
however, not all bundles of this form define such set of nodes on a sextic. It would be interesting to find a complete 
classification of Casnati--Catanese bundles of $1/2$-even sets of nodes on sextic similarly to the result \cite{CatTon}.

Theorem \ref{if_codes_ne_0} significantly simplifies the search of surfaces containing a $1/2$-even set of nodes; moreover,
we prove some useful bounds for the numbers $k$ and $m_i$, see Lemma \ref{conditions} for details. However, the list of
bundles~$\E$ is not as short as in Theorem \ref{0-even_sets}. 

If $w\subset \Sing(B)$ is an even set of nodes and $\E$ is its Casnati--Catanese bundle, we have the inequality~$\codes\ne 0$. However, 
\eqref{decomp-codes} shows that $T_2(\resx)$ can be zero; to check that this is not the case we need to compute the defect 
$d$ of $\Sing(B)$. In order to compute $d$ we relate it to the space of global sections of the twist of the sheaf 
of ideals $H^0(\p^3,\I_{\Sing(B)}(5))$. We estimate the dimension of
the space $H^0(\p^3,\I_{\Sing(B)}(5))$ using the dimensions of $H^0(\p^3,\I_{w}(5))$ for all even sets $w\in \codes$. 
Next, we compute spaces $H^0(\p^3,\I_{w}(5))$ using bundles $\E$ constructed in Theorem~\mbox{\ref{if_codes_ne_0}}. 
This leads to the main result of this paper:
\begin{theorem}\label{Result}
 If $B$ is a nodal sextic and $T_2(\resx)\ne0$, then $B = B(\E)$ is the zero locus of the determinant of an injective morphism~\mbox{\eqref{Phi}}, 
 where $\E$ is one of following bundles: 
  \begin{align*}
%   \OP_{\p^3}(-2)^{\oplus 2} &\text{ with $1/2$-even set $w(\E)$ of }27\text{ nodes};\\
  &\E = \OP_{\p^3}(-3)^{\oplus 3}\oplus\OP_{\p^3}(-2) &&\text{ with $1/2$-even set $w(\E)$ of }31\text{ nodes};\\
  &\E = \OP_{\p^3}(-2)^{\oplus 3} &&\text{ with\hphantom{1/} $0$-even set $w(\E)$ of }32\text{ nodes};\\
  &\E = \OP_{\p^3}(-3)^{\oplus 6} &&\text{ with $1/2$-even set $w(\E)$ of }35\text{ nodes};\\
  &\E = \Omega_{\p^3}^1(-1)\oplus \OP_{\p_3}(-2)&&\text{ with \hphantom{1/}$0$-even set $w(\E)$ of }40\text{ nodes}.
 \end{align*}
\end{theorem}
By Theorem \ref{Catanese_Casnati} we know that any surface $B$ described in Theorem~\mbox{\ref{Result}} contains an even set of
nodes~$w$; however, similarly to Theorem \ref{Catanese_Casnati}, if the bundle $S^2\E(\deg(B)+\delta)$ is not globally generated,
we can not be sure that $\Sing(B) = w$. 
Moreover,  such situation arises in Remark \ref{rmk_1}, for more details see 
Section~\ref{sect_conclusion}. Nevertheless, for general members of families described in Theorem \ref{Result} 
we actually have~\mbox{$\Sing(B) = w$}; this leads to such a result.
% So we get four families of examples of a sextic double solids 
% with a non-trivial obstruction to rationality.
\begin{proposition}\label{examples}
 If $\E$ is a bundle from the list in Theorem~\mbox{\ref{Result}}, then there exists a Zariski open subset 
 $\mathcal{U}$ of $H^0(\p^3,S^2\E(6+\delta))$ such that for any $\Phi\in\mathcal{U}$ in Notation \ref{notation_main} 
 we have $w = \Sing(B)$, the group $\codes$ is generated by $w$ and $T_2(\resx)\ne 0$.
\end{proposition}
Denote by \mbox{$\mathcal{Z}_{31}$, $\mathcal{Z}_{32}$, $\mathcal{Z}_{35}$} 
and $\mathcal{Z}_{40}$ the families of sextic double solids corresponding to the vector bundles listed in Theorem \ref{Result}.
\begin{remark}\label{rmk_2}
By Theorem \ref{Result} any sextic double solid with $T_2(\resx)\ne 0$ is an element of  at least 
one of the four families~\mbox{$\mathcal{Z}_{31}$, $\mathcal{Z}_{32}$, $\mathcal{Z}_{35}$} 
and $\mathcal{Z}_{40}$. 
% By Proposition \ref{examples} for a general 
% element of each of these families one has~\mbox{$T_2(\resx)\ne 0$}.
Moreover, the general element of each~$\mathcal{Z}_n$ for~\mbox{$n=31$, 32, 35} and 40 is a double cover branched along a sextic surface
with a minimal even set of $n$ nodes. Thus, these four families intersect each other by their closed subvarieties.
\end{remark}

Double solids branched along sextic surfaces with even sets of 35 or 31 nodes were studied
in~\mbox{\cite[Section 3.2]{ex-IKP}} and~\mbox{\cite{ex-Beau}} and it was proved there that they have non-trivial group~$T_2(\resx)$. 
In~\mbox{\cite[Section 3.2]{ex-IKP}} the branch surface was described as the zero locus of the determinant of a 
different vector bundle than in Theorem~\mbox{\ref{Result}}, cf.~\mbox{\cite[Theorem~2.23]{Catanese}}. 
As for sextic double solids branched along a surface with even set of $32$ and $40$ nodes, here we introduce the first 
proof that any general surface of this type admits $T_2(\resx)\ne 0$.

It happens that for all examples of non-rational sextic double solids described in Proposition~\ref{examples} 
their defects $d$ vanish, so~\mbox{\cite[Corollary A, Corollary 4.2]{Cheltsov_Park}} gives another proof of 
non-rationality of these double solids. However, using our method and \mbox{\cite[Proposition 1]{AM}} we prove 
stable non-rationality of all these varieties. Thus, Proposition~\ref{examples} introduces two new families of 
rationally connected stably non-rational varieties. 

In Section~\mbox{\ref{sect_conclusion}} we discuss a possible
example of a nodal sextic double solid where non-rationality can be proved by an Artin--Mumford obstruction but cannot be deduced from \cite{Cheltsov_Park}, 
and mention some other relevant open questions.

We point out that stable non-rationality of a very general (smooth) quartic double solid is known from~\cite{HT}.
However, the approach using the Artin--Mumford obstruction gives a more precise information about stable non-rationality for 
certain families of singular sextic double solids, cf.~\cite{Voisin}.

The paper is organized as follows: in Section~\mbox{\ref{sect_cascat}} we recall the construction~\mbox{\cite[Section 1]{CasCat_with_error}}
which helps us to associate with any even  set of nodes $w$ a special Casnati--Catanese bundle~$\E$. In 
Section~\mbox{\ref{sect_minimal_1/2_even}} we introduce useful properties of $1/2$-even sets of nodes on sextic surfaces
and prove Theorem~\mbox{\ref{if_codes_ne_0}}. In Section~\mbox{\ref{sect_defect}} we study the notion of defect. 
In Section~\mbox{\ref{sect_computations_of_defect}} 
we compute the defects of even sets of nodes with Casnati--Catanese bundles described in 
Theorem~\mbox{\ref{if_codes_ne_0}}, which helps us to prove Theorem~\mbox{\ref{Result}} in Section~\mbox{\ref{sect_proof}}. Also, this section
contains the proof of Proposition~\mbox{\ref{examples}}. In Section~\mbox{\ref{sect_conclusion}} we review an
explicit construction of a non-general subfamily of sextic double solids in the family $\mathcal{Z}_{32}$;
also there we discuss possible directions of further research.

\smallskip
\textbf{Acknowledgements.}
I am grateful to my advisor, Constantin Shramov, for suggesting this problem as well as for his patience and invaluable support.
I also thank Anton Fonarev and Lyalya Guseva for useful discussions. I was supported by the Foundation for the Advancement of Theoretical Physics and Mathematics ``BASIS''.

\section{Casnati and Catanese construction}\label{sect_cascat}
\subsection{Horrocks correspondence}

Denote by $R$ the homogeneous coordinate ring of the projective space~\mbox{$R = \mathbb{C}[x_0,x_1,x_2,x_3]$}. 
Consider the functor
\begin{align*}
 &\Gamma_* \colon \mathrm{Coh(\p^3)} \to R\grmod,\\
 &\Gamma_*(\G) =  \bigoplus_{i\in\mathbb{Z}} H^0(\p^3,\G(i)),
\end{align*}
where $\G$ is a coherent sheaf on $\p^3$ and $R\grmod$ is the category of
graded $R$-modules. We denote by $R\Gamma_*$ the right derived functor of $\Gamma_*$ between derived categories:
\begin{equation*}
 R\Gamma_* \colon \dbcoh \to \mathcal{D}^b(R\grmod).
\end{equation*}
Fix  an element $\C$ of $\dbcoh$ and two integers $r\leqslant s$,
% represented by a complex $\C^{\bullet}$ of graded modules $\C^{i}\in R\grmod$
% with a differential $\delta^i\colon \C^i\to \C^{i+1}$. 
then by $\tau_{\geqslant r}\tau_{\leqslant s}\C$ we denote the \emph{truncation} 
of~$\C$ (for definition see \cite[Section 2]{Walter}).
Moreover, by $H^i(\C)\in R\grmod$ we denote the cohomology module of $\C\in \dbcoh$ of degree $i$.
Denote by $D'$ the following full subcategory of $\mathcal{D}^b(R\grmod)$:
\begin{equation*}
 D' = \left\{ \C\in \mathcal{D}^b(R\grmod)\ |\ \dim(H^1(\C)\oplus H^2(\C)) < \infty,\ H^i(\C) = 0 \text{ for } i\ne 1,2\right\}
\end{equation*}
% Denote by $D'$ the  full subcategory of $\mathcal{D}^b(R\grmod)$, objects there are represented by complexes~$\C$ with the property~\mbox{$H^i (\C)$}
% is a module of finite dimension for $i=1,2$, and $H^i (\C)$ vanish for all $i \ne 1, 2$. 

We say that \emph{a morphism $f\colon \G'\to \G''$ between vector bundles is equivalent to zero} and write~\mbox{$f\sim 0$} if it factors through
a direct sum of line bundles. Denote by $\mathrm{Bun}_{\p_3}'$ the category which objects are vector bundles over~$\p^3$ and
morphisms are quotients by the equivalence relation:
\begin{equation*}
 \Hom_{\mathrm{Bun}_{\p_3}'}(\G',\G'') = \Hom_{\mathrm{Bun}_{\p_3}}(\G',\G'')/\sim.
\end{equation*}
Note that for any line bundle $\mathcal{L}$ on the projective space one has $\tau_{\geqslant 1}\tau_{\leqslant2}R\Gamma_* (\mathcal{L}) = 0$,
hence, the functor 
\begin{equation*}
\tau_{\geqslant 1}\tau_{\leqslant2}R\Gamma_*\colon \dbcoh \to \mathcal{D}^b(R\grmod)
\end{equation*}
factors through $\mathrm{Bun}_{\p^3}'$ and takes values in $D'$. Thus, we have a functor 
\begin{equation*}
 \tau_{\geqslant 1}\tau_{\leqslant2}R\Gamma_*\colon\mathrm{Bun}_{\p_3}' \to D'.
\end{equation*}
The following theorem states that it is an equivalence of categories.
\begin{theorem}[{\cite[Proposition 2.10]{Walter}}]\label{thm_Horrocks_corr}
 There exists an inverse functor $\Syz\colon D'\to \mathrm{Bun}_{\p_3}'$ to $\tau_{\geqslant 1}\tau_{\leqslant2}R\Gamma_*$.
 Moreover, for any bundle $\G$ over $\p^3$ we have that 
 \begin{equation}\label{eq_property_of_syz}
  \G = \Syz(\tau_{\geqslant 1}\tau_{\leqslant2}R\Gamma_*(\G))\oplus \bigoplus_{i=1}^l \mathcal{L}_i,
 \end{equation}
 where $\mathcal{L}_i$ is a line bundle on $\p^3$ for each $1\leqslant i \leqslant l$.
\end{theorem}

\subsection{Construction of $\E$}
Here we recall the construction from \cite[Section 1]{CasCat}.
Let us fix a $1/2$-even set of nodes $w$ on a nodal sextic surface~$B$. The group $\Pic(\resb)$ is torsion-free.
\begin{lemma}\label{lemma_pic_tf}
 The Picard group of the resolution $\resb$ of a nodal sextic surface $B\subset \p^3$ is torsion-free.
\end{lemma}
\begin{proof}
  By \cite[Theorem 1]{Tyurina} the resolution $\resb$ of $B$ is a deformation of a smooth sextic surface $\mathscr{B}_t$. 
  Thus, $\resb$ is diffeomorphic to $\mathscr{B}_t$. By the Lefschetz hyperplane theorem $H^2(\mathscr{B}_t, \mathbb{Z})$
  is torsion-free, then so is $H^2(\resb,\mathbb{Z})$. Finally, by exponential exact sequence the torsion subgroup of $\Pic(\resb)$ 
  maps injectively to $H^2(\resb,\mathbb{Z})$; thus, we get the result.
\end{proof}
Now we introduce some more necessary notations.
\begin{notation}\label{notation_2}
By Lemma \ref{lemma_pic_tf} in the notation of Definition \ref{def_esn} there exists 
a unique divisor class~$D_w$ on the surface $\resb$ such that $2D_w = E_w$.
eDnote by $\Xw$ the corresponding double cover  of $\resb$ branched along~\mbox{$E_w$}:
\begin{equation*}\label{eq_X_w}
\Xw \xrightarrow[2:1]{\ \pi_w} \resb \xrightarrow{\ \sigma_B\ }  B.
\end{equation*}
The direct image of $\OP_{\Xw}$ under $\pi_w$ splits into the sum of line bundles:
\begin{equation*}\label{eq_pi(O_Xw)}
{\pi_w}_*(\OP_{\Xw}) = \OP_{\resb} \oplus \OP_{\resb}(-D_w). 
\end{equation*}
 Denote by $\F$
the direct image of the second summand under $\sigma$ and the inclusion $i$ (see Notation \ref{notation_main}):
\begin{equation*} 
\label{def_of_F}
\F:=i_*{\sigma}_* \OP_{\resb}(-D_w). 
\end{equation*}
\end{notation}
We are going to construct the morphism $\Phi$ of Theorem \ref{Catanese_Casnati} from a locally free resolution of the sheaf~$\F$. 
In order to get this resolution, let us consider the following complex $\C^{\bullet}$ representing an object $\C$ of the derived category 
of graded $R$-modules (cf. \cite[p. 240]{CasCat_with_error}):
\begin{equation}\label{eq_C}
 \C^i = \left\{ \begin{aligned}
                 &\bigoplus_{j\geqslant 2} H^1(B,\F(j)), &&\text{ if } i=1;\\
                 &\hspace{40pt}0, &&\text{ otherwise}.
                \end{aligned}
 \right.
\end{equation}
Thus, $\C$ is an object in $D'$. By Theorem \ref{thm_Horrocks_corr} there exists a vector bundle $\Syz(\C)$ 
 such that we have an isomorphisms of graded modules for $i=1$ and $2$:
 \begin{equation*}
 \bigoplus_{j\in\mathbb{Z}} H^i(\p^3,\Syz(\C)(j))\cong H^i(\C).
 \end{equation*}
 Then the inclusion $\C \to \bigoplus_{i,j} H^i(\p^3,\F(j))$ induces  the morphism of sheaves:
 \begin{equation*}
  \epsilon\colon \Syz(\C) \to \F,
 \end{equation*}
 which can be non-surjective. Find a bundle $L=\bigoplus_i \OP_{\p^3}(-i)^{\oplus m_i}$ of
 minimal rank such that we have surjections
 \begin{equation*}
  \epsilon\colon (\Syz(\C) \oplus L)\otimes\OP_{\p^3} \twoheadrightarrow \F,
 \end{equation*}
 for any integer $j$. Then we can describe the bundle $\E$ and the morphism $\Phi$ from Theorem~\ref{Catanese_Casnati} explicitly. 
 \begin{theorem}[{\cite{CasCat}, \cite[Corollary 0.4]{CasCat_with_error}}]\label{thm_CC}
  Assume $w$ is a $1/2$-even set of nodes on a nodal surface $B$ of even degree and let $\E=\Syz(\C)\oplus L$ in the above notations.
  Then we have the following exact sequence
 \begin{equation}\label{es_F}
 0\to \E^{\vee}(-\deg(B) - \delta)\xrightarrow{\Phi} \E \to \F\to 0.
\end{equation}
Moreover, we have the following condition on the Chern class and the rank of $\E$:
\begin{equation}\label{eq_deg_rk}
 2c_1(\E) + \mathrm{rk}(\E)(\deg(B)+\delta) = \deg(B).
\end{equation}
\end{theorem}
 
 So by~\mbox{\cite[Corollary 0.4]{CasCat_with_error}} for any even set of nodes $w$ we can associate a unique up to an 
 isomorphism Casnati-Catanese vector bundle $\E$ of $w$. From now on we consider only Casnati--Catanese bundles obtained from this construction.

\section{Minimal $1/2$-even sets of nodes}\label{sect_minimal_1/2_even}
In this section we are going to show the important properties of minimal even sets of nodes (see Definition~\ref{def_minimal}).
Our goal here is to describe explicitly all nodal sextic surfaces with a minimal~$1/2$-even set of nodes and
prove Theorem~\mbox{\ref{if_codes_ne_0}}. The first lemma is useful for surfaces of any degree
and any minimal even set of nodes.
\begin{lemma}\label{property_of_minimality}
 Assume that $w$ is a minimal $\delta/2$-even set of nodes on a nodal surface $B$, 
 then~\mbox{$h^1(\p^3,\F(n))$} vanishes for any~\mbox{$n\leqslant 0$} and~\mbox{$n\geqslant \deg(B)-4+\delta$}.
\end{lemma}
\begin{proof}
 For any set of points $w$ denote by $2^w$ the $\mathbb{F}_2$-vector space generated by all its elements. 
 When $w$ is an even set of nodes we consider $2^w$ as a subspace of $2^{\Sing(B)}$. 
 Recall that $\codes$ is also a subspace in $2^{\Sing(B)}$, 
 generated by all even sets of nodes on $B$. 
 
 On the one hand, the minimality of the even set $w$ implies that the space $\codes\cap 2^w$ of all even subsets of nodes inside~$w$ is one-dimensional.
 On the other hand, in Notation~\ref{notation_2} by~\mbox{\cite[Lemma 2]{Beauville}} we have the following inequality 
 (cf. also~\mbox{\cite[Theorem 4.5]{Jaffe-Ruberman}}):
\begin{equation*}
 1=\dim(\codes\cap 2^w)\geqslant 1+ h^1(\Xw, \mathbb{Z})= 1+2 h^1(\Xw,\OP_{\Xw}) .
\end{equation*}
Thus, one has $h^1(\Xw,\OP_{\Xw}) = 0$. Then since the finite morphism $\pi_w$ has no higher direct images
the ranks \mbox{$h^1(B, \OP_{B})$} and $h^1(\p^3,\F)$ also vanish.

To show the same for all negative twists of $\F$, we choose a general hyperplane $H$ in $\p^3$. Then the intersection
$H\cap B$ is a smooth curve by Bertini theorem. The divisor $2D_w = E_w$ is effective by definition. Then so is the intersection $2D_w\cap H$;
thus, $\deg(2D_w\cap H)>0$. Therefore, the sheaf $\F(n)|_H$ is of strictly negative degree for all~\mbox{$n\leqslant0$}. 
Consider the exact sequence:
\begin{equation}\label{eq_es_for_F}
 0\to \F(n-1) \to\F(n)\to \F(n)|_{H}\to 0.
\end{equation}
By discussion above $h^0(\p^3,\F(n)|_H)=0$. Thus, by induction on $-n$ for all negative $n$ we prove that $h^1(\p^3,\F(n)) = 0$. 

Finally, use Serre duality on the smooth surface $\resb$. Since $B$ has du Val singularities the canonical class 
of $\resb$ is equal to~\mbox{$\sigma^{*} (\OP_{\p^3}(\deg(B)-4)|_{B})$}.
Then by \cite[Remark 2.15]{Catanese} we have
\begin{multline}\label{eq_Serre_duality}
 H^i(\p^3,\F(n)) = H^i\left(\resb,\OP_{\resb}\left(nH - D_w\right)\right) \cong
 \\ \cong H^{2-i}\left(\resb,\OP_{\resb}\left((\deg(B)-4-n)H+ D_w\right)\right) ^{\vee} \cong\\
 \cong H^{2-i}\left(\resb,\OP_{\resb}\left((\deg(B)-4+\delta-n)H- D_w\right)\right)^{\vee} = \\ =  H^{2-i}(\p^3,\F(\deg(B)-4+\delta-n))^{\vee}.
\end{multline}
Applying this isomorphism we prove the assertion of the lemma.
\end{proof}

From now on we restrict to the case when $B$ is a sextic surface and $w$ is a $1/2$-even set of nodes. 
Following Endrass \cite{Endrass99} we say that~$w$  \emph{is cut out by a plane} if there exists a plane~\mbox{$\Pi\subset \p^3$} intersecting
$B$ along a curve $C$ with multiplicity~2, so $\Pi\cdot B = 2C$ with the following property
% and~\mbox{$w=\Sing(B)\cap C$}.
\begin{equation*}
 w=\Sing(B)\cap C.
\end{equation*}
Even sets of nodes which are cut out by planes were studied in~\mbox{\cite[Proposition 2.4]{Endrass99}}, 
see also Proposition \ref{prop_es_cut_by_plane} below. 

Now assume that $w$ is not cut out by a plane. Consider the following graded $R$-module.

% , so 
% we are going to discuss all other even sets of nodes. 
% In view of the construction of the Casnati--Catanese bundles we are interested in the following $R$-module
\begin{equation}\label{M}
 M=\bigoplus_{n\in \mathbb{Z}} H^0(\p^3,\F(n)).
\end{equation}
By the next assertion the module $M$ is generated by its subspace:
\begin{lemma}\label{m_2_m_3_m_4}
  If $w\in \codes$ is a $1/2$-even set of nodes which not cut out by a plane, then the graded module~\mbox{$M$} 
  is generated by its graded components~\mbox{$H^0(\p^3,\F(2))$, $H^0(\p^3,\F(3))$}, and~$H^0(\p^3,\F(4))$.
 \end{lemma}
 \begin{proof}
 By Lemma~\mbox{\ref{property_of_minimality}} we have $h^1(\p^3,\F(n)) = 0$ for $n\ne 1,\ 2$. 
 Since $w$ is not cut out by a plane by \cite[p. 89]{Endrass98} we have $h^0(\p^3,\F(1))= 0$.
 Thus, by \eqref{eq_es_for_F} one has $h^0(\p^3,\F(n))= 0$ for all~\mbox{$n\leqslant 1$}. By \eqref{eq_Serre_duality} we get:  
 \begin{equation*}
   h^2(\p^3,\F(n))= h^0(\p^3,\F(3-n))=0,
 \end{equation*}
 for all $n\geqslant 2$.
 Finally, since the support of $\F$ is a surface $B$, all groups $H^3(\p^3,\F(n))$ vanish for all integers $n$.
 Therefore, $\F(4)$ is Castelnuovo--Mumford regular 
 in the sense of~\mbox{\cite[Lecture 14]{Mumford}}). Then by loc. cit. we get the assertion.
 \end{proof}
 Now using Lemma~\mbox{\ref{m_2_m_3_m_4}} 
 we can describe the  Casnati--Catanese vector bundles $\E$ of all the~\mbox{$1/2$-even} minimal sets on a sextic surface. 
 \begin{proposition}\label{description_of_E_1_2_even}
   If $w\in \codes$ is a minimal $1/2$-even set of nodes, then there exists a Casnati--Catanese 
   bundle which is isomorphic to
   \begin{equation*}
    \E  = {\Omega_{\p^3}^1(-2)} ^{\oplus k} \oplus \left( \bigoplus \OP_{\p^3}(-i)^{\oplus m_i}\right),
   \end{equation*}
   for some non-negative integers $k$ and $m_i$ depending only on $w$. 
   Moreover, if~$w$ is not cut out by a plane, then 
   \begin{equation}\label{eq_form_of_CC_bundle}
   \E = ({\Omega_{\p^3}^1(-2)}\otimes W) \oplus (\OP_{\p^3}(-2)\otimes U_2)\oplus (\OP_{\p^3}(-3)\otimes U_3) \oplus (\OP_{\p^3}(-4)\otimes U_4).
   \end{equation}    
   where $W$ and $U_i$ are respectively $k = k(w)$-dimensional and $m_i = m_i(w)$-dimensional vector spaces. Also in this case we have such equalities:
   \begin{equation*}
    \begin{split}
   & h^0(\p^3,\F(2)) = m_2 + k;\\
   & h^1(\p^3,\F(2)) = k;\\
   & h^0(\p^3,\F(3)) = m_3, \text{ if } m_2=0.     
   \end{split}
   \end{equation*}
   More generally, $H^0(\p^3,\F(3))$ is generated by the image of $H^0(\p^3,\F(2))\otimes H^0(\p^3,\OP_{\p^3}(1))$ in it and $m_3$ additional generators.
 \end{proposition}
 \begin{proof}
 We follow the construction in Section~\mbox{\ref{sect_cascat}}. By Lemma \ref{property_of_minimality} and using the notation \eqref{eq_C}
 we have:
 \begin{equation*}
  H^1(\C^{\bullet}) = H^1(\p^3,\F(2))
 \end{equation*}
 as graded $R$-modules. Denote by $W$ the vector space $H^1(\p^3,\F(2))$. All other cohomologies of $\C$ vanish. 
 Note that the object $W\otimes R\Gamma_*(\Omega^1_{\p^3}(-2))$ of the derived category of $R$-modules 
 has the same $R$-module structure on its intermediate cohomologies. Thus, the isomorphism between objects of $D'$:
 \begin{equation*}
  \C \cong W\otimes R\Gamma_*(\Omega^1_{\p^3}(-2)).	
 \end{equation*}
 Since the bundle $W\otimes \Omega^1_{\p^3}(-2)$ is semistable and its slope is not an integer, it does 
 not split into a direct sum of a line bundle and some other vector bundle. Thus, by Theorem \ref{thm_Horrocks_corr}
 we have an isomorphism of vector bundles over $\p^3$:
 \begin{equation*}
  \Syz(\C) \cong W\otimes \Omega^1_{\p^3}(-2).
 \end{equation*}
 Thus, we get the first assertion of the proposition. 
 In case when $w$ is not cut out by a plane since~\mbox{$H^0(\p^3,\F(2))$, $H^0(\p^3,\F(3))$}, and~$H^0(\p^3,\F(4))$ 
 generate the module $M$ by Lemma \ref{m_2_m_3_m_4}. Thus, if we have a surjection
 \begin{equation*}
    \E  = {\Omega_{\p^3}^1(-2)} ^{\oplus k} \oplus \left( \bigoplus \OP_{\p^3}(-i)^{\oplus m_i}\right) \to \F,
 \end{equation*}
 then its restriction on the greatest subbundle of the form \eqref{eq_form_of_CC_bundle} is also surjective.
 Therefore, the construction in Theorem \ref{thm_CC} gives us the bundle of this form and of minimal possible rank.
 Conditions on numbers $h^i(\p^3,\F(j))$, $k$ and $m_k$ follows from the computation of cohomologies of the cotangent 
 and line bundles on $\p^3$.
\end{proof}

 Theorem \ref{if_codes_ne_0} immediately follows from Proposition \ref{description_of_E_1_2_even}
 Moreover, we see that the Casnati--Catanese bundles of minimal $1/2$-even sets of nodes which are not cut out by planes depend only 
 on the four numbers $k = k(w)$ and $m_i = m_i(w)$. Below we introduce some additional restrictions on them.
\begin{lemma}\label{conditions}
  Suppose that $B$ is a sextic nodal surface, $w\in\codes$ is minimal $1/2$-even set of nodes which is not cut out by a plane.
  Then 
  \begin{enumerate}
   \item[$(1)$] $m_2= k + \frac{35-|w|}{4}$;
   \item[$(2)$] $m_3\leqslant 6-k$;
   \item[$(3)$] $k+3m_2+m_3 - m_4 = 6$.
  \end{enumerate}
 \end{lemma} 
 \begin{proof}
 The sheaf  $R^1\sigma_*\OP_{\resb}(-D_w) = 0$; thus,
%  Cascat page 248...
 by the Riemann--Roch theorem on the smooth surface $\resb$ applied to the line bundle $\OP_{\resb}(-D_w)$ we have:
  \begin{multline}\label{eq_Riemann-Roch}
   \chi(\F(n)) = \chi(nH-D_w) = \\= \frac{1}{2}\cdot \left(nH - \frac{ E_w}{2}\right)\cdot \left(nH-\frac{E_w}{2}- K_{\resb} \right) + \chi(\OP_{\resb}) =\\
   = \frac{44+3(2n-\delta)(2n-4-\delta) -|w|}{4}.
  \end{multline}
  Here we use the fact that the class $D_w$ is numerically equivalent to the class $E_w/2$.
 By Proposition \ref{description_of_E_1_2_even} one has that $m_2=h^0(\p^3,\F(2))$ and $k=h^1(\p^3,\F(2))$. Then the first assertion 
 of the lemma follows from the formula \eqref{eq_Riemann-Roch} for $n=2$.
 
  By Proposition \ref{description_of_E_1_2_even} the group $H^0(\p^3,\F(3))$ contains at least $m_2$-dimensional space 
  generated by the previous component of the graded module $M$ (see \eqref{M}).   
  Then the number $m_3$ is less than or equal to $h^0(\p^3,\F(3))-m_2$. 
  The minimality implies that $h^1(\p^3,\F(3))=0$ by Lemma \ref{property_of_minimality}.
  Since $w$ is not cut out by plane by \cite[p. 89]{Endrass98} and \eqref{eq_Serre_duality} we have also~\mbox{$h^2(\p^3,\F(3))=0$}.
  Finally, $h^3(\p^3,\F(3))=0$ since the support of $\F$ is a surface; thus, we get
  \begin{equation*}
   m_3 \leqslant \chi(\F(3))-m_2.
  \end{equation*}
  Thus, by formula \eqref{eq_Riemann-Roch} for $n=2$ and $3$ we get the second assertion.
  
 Finally, the formula \eqref{eq_deg_rk} implies the last assertion.
\end{proof}

Now let us show one more consequence of the Riemann--Roch theorem.
\begin{lemma}\label{lemma_RR}
 Assume that $w$ is a $\delta/2$-even minimal set of nodes on a nodal sextic surface~$B$ 
 and $w$ is not cut out by a plane. Then
 \begin{equation*}
  h^0(\p^3,\F(4)) = h^0(\p^3,\F(3))+12
 \end{equation*}
\end{lemma}
\begin{proof}
  This follows by the  Riemann--Roch formula \eqref{eq_Riemann-Roch} for $n=3$ and $4$ and Lemma \ref{property_of_minimality}.
%  \begin{align*}
%   &\chi(\F(3)) = \frac{59-|w|}{4};\\
%   &\chi(\F(4)) = \frac{107-|w|}{4}.\\
%  \end{align*}
% By Lemma \ref{property_of_minimality} we have that $\chi(\F(i)) = h^0(\p^3,\F(i))$ for any $i>2$; therefore, we get the result.
\end{proof}

 In order to get additional conditions on numbers $k$ and $m_i$, we write the Casnati--Catanese bundle~$\E$ 
 from \eqref{eq_form_of_CC_bundle} as  a direct sum of four bundles
\begin{equation}\label{eq_e_i}
 \E 
%  = \left(W \otimes{\Omega_{\p^3}^1(-2)}\right) \oplus\left( U_2 \otimes\OP_{\p^3}(-2)\right)\oplus \left(U_3 \otimes \OP_{\p^3}(-3)\right) \oplus \left(U_4 \otimes \OP_{\p^3}(-4)\right)
  = \E_1 \oplus \E_2\oplus \E_3 \oplus \E_4,
\end{equation}
where $\E_1 = W \otimes{\Omega_{\p^3}^1(-2)}$
and~\mbox{$\E_i = U_i \otimes\OP_{\p^3}(-i)$} for $i=2, 3,4$. 
This decomposition induces the decomposition of the morphism $\Phi=(\Phi_{ij})$. Thus, $\Phi$ is a symmetric $4\times 4$ matrix of blocks.
Each block $\Phi_{ij}$ corresponds to an element in~\mbox{$H^0(\p^3,\E_i\otimes\E_j(7))$} if~$i\ne j$,
or in~\mbox{$H^0(\p^3,S^2\E_i\otimes\OP_{\p^3}(7))$}, if~$i=j$ and induces the morphism:
\begin{equation*}
 \Phi_{ij}\colon E_i^{\vee}(-7)\to E_j.
 \end{equation*}
 Moreover, by Theorem \ref{Catanese_Casnati} the morphism $\Phi$ is symmetric; thus, $\Phi_{ij} = \Phi_{ji}^t$.
\begin{lemma}\label{lemma_cond_on_m_4}
 Suppose $B$ is a nodal sextic surface, $w\in\codes$ is minimal $1/2$-even set of nodes, and $\E$ is its Casnati--Catanese bundle.
 If $m_2 = 1$, then the morphism
 \begin{equation*}
  \Phi_{43}\colon U_4^{\vee} \to U_3
 \end{equation*}
 is injective; in particular $m_3\geqslant m_4$.
\end{lemma}
\begin{proof}
 Assume that $m_2=1$ and consider the explicit form of the morphism $\Phi$:
%  {\small
 \begin{multline*}
  T_{\p^3}(-5)^{\oplus k}\oplus \OP_{\p^3}(-5) \oplus \OP_{\p^3}(-4)^{\oplus m_3} \oplus \OP_{\p^3}(-3)^{\oplus m_4}  \xrightarrow{\Phi} \\ \xrightarrow{\Phi}
  \Omega^1_{\p^3}(-2)^{\oplus k} \oplus \OP_{\p^3}(-2) \oplus \OP_{\p^3}(-3)^{\oplus m_3} \oplus \OP_{\p^3}(-4)^{\oplus m_4}. 
 \end{multline*}
%  }
 
 Since $\Phi_{41}$ corresponds to an element in~\mbox{$H^0(\p^3,\E_1\otimes\E_4(7)) = H^0(\p^3,\Omega_{\p^3}^1(1)^{\oplus km_4})=0$}; thus, we have~\mbox{$\Phi_{41} = 0$}. 
 Also, $\Phi_{44}$ is induced by~\mbox{$H^0(\p^3,S^2\E_4(7)) = H^0(\p^3,\OP_{\p^3}(-1)^{m_4(m_4+1)/2}) = 0$} and this implies 
 that~$\Phi_{44}=0$. Therefore, $\Phi$ is of the following form.
 \begin{equation*}
  \Phi = \begin{pmatrix}
          \Phi_{11} &\Phi_{21} &\Phi_{31} & 0 \\
          \Phi_{12} &\Phi_{22} &\Phi_{32} & \Phi_{42} \\
          \Phi_{13} &\Phi_{23} &\Phi_{33} & \Phi_{43} \\
          0 & \Phi_{24} & \Phi_{34} & 0
         \end{pmatrix}
 \end{equation*}
%  Since $\Phi$ is an injective morphism, its restriction to the summand $\E_4^{\vee}(-7) = \OP(-3)^{\oplus m_4}$ maps it injectively 
%  to~\mbox{$\E_2\oplus \E_3 = \OP(-2) \oplus \OP(-3)^{\oplus m_3}$}. 
%  Thus, the rank of $\E_4$ is less than or equal to the rank of $\E_2\oplus\E_3$:
%  \begin{equation*}
%   m_4\leqslant 1+m_3.
%  \end{equation*}
 Assume that $\Phi_{43}$ is not injective. Then there exists $\xi\in U_4^{\vee}$ such that 
 \begin{equation*}
  \Phi_{41}(\xi) =\Phi_{43}(\xi) = \Phi_{44}(\xi) = 0.
 \end{equation*}
 Let us consider the matrix $\Phi$ in the basis containing $\xi$ as one element. Since $m_2=1$ the only 
 non-zero entry in its column is a linear form $\Phi_{42}(\xi)\in H^0(\p^3,\OP_{\p^3}(1))$. Thus, 
 the determinant of $\Phi$ is divided by $\Phi_{42}(\xi)$, but it contradicts to 
 the irreducibility of $B$. 
%  
%  
%  The block $\Phi_{42}$ is just a column. Moreover, we can choose the basis in which  
%  the squared matrix $(  \Phi_{42}, \Phi_{43} )$ is of the following form:
% \begin{equation*}
%  (  \Phi_{42}, \Phi_{43} ) = \begin{pmatrix} 
%                                                             l &l_1 &  \ldots & l_{m_3} \\
%                                                             0 & \lambda_{11} & \ldots & \lambda_{1m_3} \\
%                                                             \vdots & \vdots & \ddots & \vdots \\
%                                                             0 & \lambda_{m_3 1}& \ldots & \lambda_{m_3m_3}
%                                                            \end{pmatrix}.
% \end{equation*}
%  Here $l,l_1,\dots,l_{m_3}$ are some linear forms on $\p^3$,
%  and the matrix $(\lambda_{ij})$ is the map from $\OP(-3)^{\oplus m_4}$ to $\OP(-3)^{\oplus m_3}$. Since our field is algebraically closed, in some basis 
%  this matrix is upper-triangular.  
%  In these notations if we compute the determinant of $\Phi$, we see that $\det(\Phi)$ is divided by~$l$. However, it is impossible since $B$ is 
%  irreducible. 
%  
%  Finally, if $\Phi_{43}$ is not injective, then there exists
\end{proof}
 Recall that $\p^3 = \p(V)$ is a projectivization of a 4-dimensional vector space $V$ and $\F$ is as in Notation \ref{notation_2}.
\begin{lemma}\label{lemma_m_2=1_then_inclusion}
 If $B$ is a nodal sextic surface, $w$ is a minimal $1/2$-even set of nodes on $B$ which is not cut out by a plane and
 $m_2= 1$, then the canonical morphism 
 \begin{equation*}
\gamma\colon H^0(\p^3,\F(2))\otimes V^{\vee} \to H^0(\p^3,\F(3))  
 \end{equation*}
 is an embedding. 
\end{lemma}
\begin{proof}
If $\E$ is a 
Casnati--Catanese bundle of $w$, then $\gamma$ extends to a morphism between exact sequence of cohomologies of twists of the exact sequence \eqref{Phi}:
{\small
\begin{equation*}
 \xymatrix@C=1em{
  0 \ar[r]& H^0(\p^3,\E^{\vee}(-5))\otimes V^{\vee} \ar[r] \ar[d] & H^0(\p^3,\E(2))\otimes V^{\vee} \ar[r] \ar[d]^{\beta} & H^0(\p^3,\F(2))\otimes V^{\vee} \ar[d]^{\gamma} \ar[r] &  H^1(\p^3,\E^{\vee}(-5))\otimes V^{\vee} \ar[d]\\
  0 \ar[r]& H^0(\p^3,\E^{\vee}(-4))\ar[r] & H^0(\p^3,\E(3)) \ar[r] & H^0(\p^3,\F(3)) \ar[r] &  H^1(\p^3,\E^{\vee}(-4))
 }
\end{equation*}}
From \eqref{eq_form_of_CC_bundle} we compute
\begin{align*}
 H^1(\p^3,\E^{\vee}(-4)) &=  H^0(\p^3,\E^{\vee}(-5)) =0;  && H^0(\p^3,\E^{\vee}(-4)) =  U_4^{\vee}; \\
  H^0(\p^3,\E(3)) &=  U_2\otimes V^{\vee}\oplus U_3; && H^0(\p^3,\E(2)) = U_2;
\end{align*}
hence, the above diagram takes the following form:
% By Proposition \ref{description_of_E_1_2_even} groups $H^1(\p^3,\E^{\vee}(-5))$, $H^0(\p^3,\E^{\vee}(-5))$ and $H^1(\p^3,\E^{\vee}(-4))$ vanish. 
% Computing cohomology groups we get that the diagram is as follows:
\begin{equation*}
 \xymatrix{
  & 0 \ar[r] \ar[d] & U_2\otimes V^{\vee} \ar@{=}[r]^{\alpha\ \ \ \ \ } \ar[d]^{\beta} & H^0(\p^3,\F(2))\otimes V^{\vee} \ar[d]^{\gamma} \ar[r] & 0 \\
  0 \ar[r]& U_4^{\vee} \ar[r]^{\varphi \hspace{25pt}} & U_2\otimes V^{\vee}\oplus U_3 \ar[r]^{\psi} & H^0(\p^3,\F(3)) \ar[r] & 0 
 }
\end{equation*}
Note that the map $\beta$ is the embedding of the first summand. By diagram chasing we deduce that the kernel of $\gamma$ is
isomorphic to the kernel of the map~\mbox{$U_4^{\vee}\to U_3$} obtained as the composition of $\varphi$  with the projection to the second 
summand. This map was known to be injection by Lemma \ref{lemma_m_2=1_then_inclusion}; hence, $\Ker(\gamma) = 0$.
% Assume that there exists $x\in H^0(\p^3,\F(2))\otimes V^{\vee}$ is in the kernel of $\gamma$. Since $\alpha$ is an isomorphism of vector spaces
% the element~\mbox{$\alpha^{-1}(x)=y\in U_2\otimes V^{\vee}$} is in the kernel of $\psi\circ \beta$. Since the lower sequence is exact
% there exists $z\in U_4^{\vee}$ such that
% \begin{equation}\label{eq_2}
%  \varphi(z) = \beta(y).
% \end{equation}
% On the one hand,
% $\varphi$ is induced by the twist of the morphism $\Phi$. Thus we have
% \begin{equation*}
%  \varphi = \Phi_{42}+\Phi_{43},
% \end{equation*}
% and by Lemma \ref{lemma_cond_on_m_4}, the map $\Phi_{43}$ is injective. 
% On the other hand, $\beta$ is an embedding of the first direct summand to the sum $U_2\otimes V^{\vee}\oplus U_3$, so the composition of
% $\beta$ with the projection to $U_3$ is zero. Thus, the equality~\eqref{eq_2} is possible only if $z=y=0$. 
% Therefore,~$x$ is also zero and $\gamma$ is an embedding.
\end{proof}

\section{Defect}\label{sect_defect}
In this section we give a geometric description of the number $d$ in the 
formula~\mbox{\eqref{decomp-codes}} and study its properties. We begin with a definition:
\begin{defin}\label{def_of_defect}
 For any finite set of points~$w=\{p_1,\dots,p_n\}$ in $\p^3$  and any integer $N$  \emph{the~\mbox{$N$-de}\-fect of~$w$} is
 the following number
  \begin{equation*}
   d_N(w):= \dim \left(H^0(\p^3,\I_{w}(N)\right) - (\dim\left(H^0(\p^3,\OP_{\p^3}(N))\right) - |w|),
  \end{equation*} 
  where $\I_w$ is the ideal of the set $w$ in $\p^3$.
\end{defin}

The notion of the $N$-defect is important for us in view of the following theorem.
\begin{theorem}[{\cite[Section 3]{Clemens}}]
\label{Clemens}
 Let $B$ be a nodal surface of even degree in $\p^3$ and $N =3 \deg(B)/2-4 $. Then in the notation of formula \eqref{decomp-codes}
 we have $d = d_{N}(\Sing(B))$. 
\end{theorem}
In view of our setup we are mostly interested in the case $N = 5$. In this situation we have the following interpretation of the 
defect of a set of points.
\begin{lemma}\label{lemma_d=h^1}
 If $N = 5$ then for any set of points $w\subset \p^3$ we have $h^i(\p^3,\I_w(5)) = 0$ for $i>1$ and
 \begin{equation*}
  d(w) = h^1(\p^3,\I_{w}(5)) = h^0(\p^3,\I_{w}(5))- 56 + |w|.
 \end{equation*}
\end{lemma}
\begin{proof}
Let us consider the following exact sequence:
\begin{equation*}
0\to \I_{w}(5) \to \OP_{\p^3}(5) \to \OP_{w}\to 0.
\end{equation*}
Since $h^0(\p^3,\OP_{\p^3}(5))=56$ and $h^i(\p^3,\OP_{\p^3}(5))=0$ for $i>0$, the long exact sequence of cohomologies gives the result.
\end{proof}

Nevertheless, in this section we are going to prove an assertion for any value of $N$.
To simplify the notation, from now on we fix some number $N$ and write $d(w)$ and use the word ``defect'' instead of  ``$N$-defect'' $d_N(w)$.
Our goal is to show that if the defects of all special subsets (such as even sets of nodes in the set of all nodes) are large 
enough, then the defect of the whole set is also large. 

Denote by $\ev_i$ the evaluation function
in a point $p_i$. Then for a set $w=\{p_1,\dots,p_n\}$  we have the following map:
\begin{equation}\label{eq_ev}
 \ev_w = \bigoplus_{i=1}^n \ev_i\colon\mathbb{C}^n \to H^0(\p^3,\OP_{\p^3}(N))^{\vee}.
\end{equation}
This map gives us another useful interpretation of the defect.
\begin{lemma}\label{lemma_interpretation_of_def}
 We have an equality $d(w) = \dim(\Ker(\ev))$.
\end{lemma}
\begin{proof}
 Note that the image of $\ev$ is a dual space to $H^0(\p^3,\I_{w}(N)$, since it consists of all functions on $\p^3$ of 
 degree $N$ vanishing at all points $p_1,\dots,p_n$. Thus, the result follows from Definition \ref{def_of_defect}.
\end{proof}
Lemma \ref{lemma_interpretation_of_def} allows us to compare the defect of a set of points and the defect of its subset.
\begin{lemma}\label{defect_of_subsets}
 If $w'$ is a subset of a set of points $w$, then $d(w')\leqslant d(w)$.
\end{lemma}
\begin{proof}
 Since $\mathbb{C}^{w'}\subset \mathbb{C}^w$, we have an embedding $\Ker(\ev_{w'})\subset \Ker(\ev_w)$. Thus,
 the result follows by Lemma \ref{lemma_interpretation_of_def}.
\end{proof}

Lemma \ref{lemma_interpretation_of_def} reduced the questions about defects to linear algebra questions.
This allows us to consider the following purely algebraic setup.
\begin{notation}
 Consider the $n$-dimensional complex vector space $\mathbb{C}^n$ with a fixed basis $e_1,\dots,e_n$. 
 Denote the set of indices $[n] = \{1,\dots,n\}$. For any subset $W \subset [n]$ define the subspace of $\mathbb{C}^n$:
 \begin{equation*}
  \mathbb{C}^W = \bigoplus_{i\in W} \mathbb{C}\langle e_i \rangle.
 \end{equation*}
 Consider subsets $w_1,\dots,w_m$ of the set $[n]$. They can be considered as vectors 
 in the vector space $\mathbb{F}_2^n$. Then for any $I\subset \{1,\dots,m\}$
 we define by $w_I$ the following sum  
 \begin{equation*}
  w_I = \sum_{i\in I}w_{i}.
 \end{equation*}
 As a set $w_I$ is
 the symmetric difference of $w_i$ for all $i\in I$. 
%  Thus, with each $I$ we connect a new subset of indices $w_I$.
  
 Finally, for any subset $J\subset \{1,\dots,m\}$ we define the following subspace of $\mathbb{C}^n$:
 \begin{equation*}
  V_J = \mathbb{C}^{(\bigcap_{j\in J}w_j\setminus\bigcup_{k\not\in J} w_k)}.
 \end{equation*}
\end{notation}
\begin{remark}\label{rmk_property_of_VJ}
 Note that for any $j\in [n]$ there exists a unique $J$ such that $e_j\in V_J$.
\end{remark}
Now we can formulate the main technical assertion of this section.
\begin{lemma}\label{lemma_red_to_algebra}
  Assume that vectors $w_1,\dots,w_m$ generates an $m$-dimensional subspace of $\mathbb{F}_2^n$ and $K\subset \mathbb{C}^n$
  is a subspace such that for any non-empty $I\subset\{1,\dots,m\}$ we have 
  \begin{equation*}
   K\cap \mathbb{C}^{w_I}\ne 0.
  \end{equation*}
  Then $\dim(K) \geqslant m$. 
\end{lemma}
\begin{proof}
 Assume that $\dim(K) = k<m$. Then there exists a subset of indices $Z \subset[n]$ such that $|Z| = n-k$ and
 \begin{equation*}
  \mathbb{C}^Z \cap K = 0.
 \end{equation*}
 Denote by $S$ the complement to $Z$ in $[n]$. By Remark \ref{rmk_property_of_VJ} for each $s\in S$ there exists 
 a unique subset $J_s\subset\{1,\dots,m\}$ such that $e_{s} \in V_{J_s}$. Therefore, for each $s\in S$ we have
 \begin{equation*}
  \mathbb{C}^Z\cap V_{J_s}\ne V_{J_s}.
 \end{equation*}
 Now, since the space generated by $w_1,\dots,w_m$ in $\mathbb{F}_2^n$ is $m$-dimensional we can choose 
 an element $w_I$  in the space $\langle w_1,\dots,w_m\rangle$ such that
 \begin{equation*}
  \mathbb{C}^{w_I}\cap \bigoplus_{s\in S} V_{J_s} = 0.
 \end{equation*}
 Then $\mathbb{C}^{w_I}\subset \mathbb{C}^Z$, since otherwise, $j_s\in w_I$ for some $s\in S$ and $\mathbb{C}^{w_I}\cap V_{J_s} \ne 0$.
 However, this contradicts to the assumption that assumption $\mathbb{C}^{w_I}\cap K \ne 0$.
\end{proof}

Finally, we are ready to state the main property of defects of even sets of nodes.
\begin{corollary}\label{codes_and_defect}
 If $B$ is a nodal surface in $\p^3$ and for any non-zero even set of nodes $w\in\codes$ one has~$d(w)\geqslant 1$ then~$d(\Sing(B))\geqslant \dim(\codes)$.
\end{corollary}
\begin{proof}
 Consider the vector space $\mathbb{C}^n = \mathbb{C}^{\Sing(B)}$ generated by all singularities of $B$.
 This is a vector subspace with a canonical basis indexed by the set $[n]$. Each even set of nodes $w\in \codes$ corresponds
 to a subset of $[n]$.
%  By Lemma \ref{lemma_interpretation_of_def} for any $0\ne w \in \codes$ the kernel $\Ker(\ev_w)$ of the map \eqref{eq_ev}
%  is non-zero. 
 Denote by $K$ the kernel of the map $\ev_{\Sing(B)}$ defined in~\eqref{eq_ev}.
 
 Then for any non-zero $w\in \codes$ the space $K$ intersects with $\mathbb{C}^w$ by the subspace $\Ker(\ev_w)$ 
 which is non-zero by the assumption and Lemma \ref{lemma_interpretation_of_def}.
 Thus, by Lemma \ref{lemma_red_to_algebra} we get the result.
\end{proof}

\section{Defect of minimal $1/2$-even sets of nodes}\label{sect_computations_of_defect}
In this section we are going to describe the conditions on Casnati--Catanese bundle of the even sets of
nodes~$w$ which arise if we assume that the defect of $w$ vanishes. From now on we consider only 
nodal sextic surfaces. In particular, in the notation of Theorem \ref{Clemens} we fix~\mbox{$N=5$}; thus, 
the defect is the $5$-defect. 

By the next assertion we can restricts to only those even sets of nodes which 
are not cut out by a plane.
\begin{proposition}[{\cite[Page 6]{Endrass99}}]\label{prop_es_cut_by_plane}
 If $w$ is cut out by a plane, then $d(w)>0$.
\end{proposition}

Now let us consider the case of minimal $1/2$-even sets of nodes. Their Casnati--Catanese bundles are described
in Proposition \ref{description_of_E_1_2_even}. If we assume that the defect of such a set of nodes 
vanishes, then we get the condition 
on numbers $k$ and $m_i$ arising in the following proposition.
\begin{proposition}\label{prop_techlemma}
Assume that $B$ is an irreducible nodal sextic surface and $w$ is a minimal~\mbox{$1/2$-even} set of nodes on $B$ and~\mbox{$d(w)=0$}. Then,
in the notation of Proposition $\ref{description_of_E_1_2_even}$, we have that either
\begin{enumerate}
  \item[$(1)$] $m_2=0$ and $k = 0, 4,5$ or $6$;
  \item[$(2)$] $m_2=1$ and $k=0$.
\end{enumerate}
\end{proposition}

In view of Lemma \ref{lemma_d=h^1} we are interested in the cohomologies of the sheaf $\I_w(5)$. To compute them
we will use the following free resolution $F_{\bullet}\to\I_w(5)$ of this sheaf. By $\spl(\E)$ here we denote the
bundle of traceless endomorphisms of $E$. 
\begin{lemma}\label{es_for_Iw}
Assume $B$ is a nodal sextic surface with a $\delta/2$-even set of nodes $w$ and $\E$ is a Casnati--Catanese bundle of $w$. 
Then we have the following exact sequence:
\begin{equation}\label{eq_es_for_Iw}
 0\to \Lambda^2 \E^{\vee}(-7-\delta) \xrightarrow{\psi} \spl(\E)(-1)\xrightarrow{\phi}  S^2 \E(5+\delta) \xrightarrow{\varepsilon} \I_w(5) \to 0.
\end{equation}
% where $(n_1,n_2,n_3) =(5,-1,-7)$ or $(6,-1,-8)$ for $\delta=0$ and $1$ respectively. 
Morphism $\varepsilon$ is the cokernel of $\phi$;
morphisms~$\psi$ and $\phi$ are the following compositions:
\begin{align*}
 &\psi:\Lambda^2 \E^{\vee}(-7-\delta) \hookrightarrow \E^{\vee}\otimes \E^{\vee} (-7-\delta) \xrightarrow{\Phi\otimes \mathrm{Id}_{\E}} \E\otimes \E^{\vee}(-1) \twoheadrightarrow
 \frac{\E\otimes \E^{\vee}}{\OP_{\p^3}}(-1); \\
 &\phi: \spl(\E)(-1) \hookrightarrow  \E\otimes \E^{\vee}(-1) \xrightarrow{ \mathrm{Id}_{\E}\otimes \Phi(5+\delta)} 
 \E\otimes\E(5+\delta) \twoheadrightarrow S^2 \E(5+\delta).
\end{align*}
\end{lemma}
\begin{proof}
 Since $\E$ is a Catanese-Casnati bundle by Definition \ref{def_CC} and Theorem \ref{Catanese_Casnati}
 the even set of nodes $w$ as a scheme coincide with $\cork(\Phi)\geqslant 2$. Thus, locally the exact sequence \eqref{eq_es_for_Iw}
 looks like the sequence $\mathbb{L}(X)$ in  \cite{Jozefiak}. Thus, the result by follows by~\cite[Theorem 3.1]{Jozefiak}.

 Moreover, we know, that 
 \begin{equation*}
 n_1-n_2 = n_2-n_3 = \deg(\Phi) = 6 + \delta.
 \end{equation*}
 Since $\deg(\I_w(5)) = 5$, we can compute all the numbers $n_1$, $n_2$ and $n_3$.
\end{proof}

Remind the notations $W$ and $U_i$ from  \eqref{eq_form_of_CC_bundle}. 
In terms of them we can describe cohomologies of the bundles~$F_i$ (here $V$ is as in Notation \ref{notation_main}):
\begin{lemma}\label{lemma_cohomologies_F_i}
 The bundles $S^2\E(6)$, $\spl(\E)(-1)$ and $\Lambda^2 \E^{\vee}(-8)$  have the following cohomology groups:
 \begin{equation*}
 \begin{split}
  &H^{0}(\p^3,S^2\E(6)) = \left(W\otimes U_2 \otimes \Lambda^2V^{\vee}\right)\oplus 
        \left(S^2 U_2 \otimes S^2V^{\vee}\right)
        \oplus  \left(U_2\otimes U_3 \otimes V^{\vee}\right)	 \\ 
   & \oplus S^2 U_3\oplus (U_2\otimes U_4);\\
    &H^{1}(\p^3,S^2\E(6)) =  \left(S^2 W \otimes \Lambda^2V^{\vee})\right)
    \oplus \left(W\otimes U_4\right); \\   
  &H^{0}(\p^3,\spl(\E)(-1)) = \left(U_2\otimes W^{\vee}\otimes V\right) 
  \oplus ( U_2\otimes U_3^{\vee})\oplus (U_3\otimes U_4^{\vee})	
  \oplus \left(U_2\otimes U_4^{\vee}\otimes V^{\vee}\right);\\
  &H^{1}(\p^3,\spl(\E)(-1)) = \left( W\otimes W^{\vee}\otimes V)\right)  
 \oplus (W\otimes U_3^{\vee}); \\ 
 &H^{0}(\p^3,\Lambda^2 \E^{\vee}(-8))  = \Lambda^2 U_4^{\vee} ;\\
 &H^{1}(\p^3,\Lambda^2 \E^{\vee}(-8))  =   S^2W^{\vee}; \\
  &H^{2}(\p^3,\Lambda^2 \E^{\vee}(-8))  = U_2^{\vee}\otimes W^{\vee}; \\
  &H^{3}(\p^3,\Lambda^2 \E^{\vee}(-8))  =\Lambda^2 U_2^{\vee}. \\
  \end{split}
 \end{equation*}
 All other cohomology groups vanish.
\end{lemma}
\begin{proof}
By \eqref{eq_form_of_CC_bundle} the bundles $S^2\E(6)$, $\spl(\E)$ and $\Lambda^2 \E^{\vee}(-8)$ split into a direct sum of tensor 
powers of the tangent bundle and some line bundles of $\p^3$ (here for simplicity we 
write~\mbox{$\OP$, $T$} and~$\Omega^1$ instead of $\OP_{\p^3}$, $T_{\p^3}$ and $\Omega^1_{\p^3}$).
\begin{align*}
 S^2\E(6) &= (S^2 W \otimes S^2\Omega^1(2))\oplus (\Lambda^2 W \otimes \Lambda^2\Omega^1(2)) \oplus (W\otimes U_2\otimes \Omega^1(2))
 \\&\oplus (W\otimes U_3\otimes \Omega^1(1))
 \oplus (W\otimes U_4\otimes \Omega^1)
 \oplus (S^2U_2\otimes\OP(2))
  \oplus (U_2\otimes U_3\otimes\OP(1))
 \\&\oplus (S^2U_3\otimes\OP) 
 \oplus (U_2\otimes U_4\otimes\OP) 
 \oplus (U_3\otimes U_4\otimes\OP(-1))
 \oplus (S^2U_4\otimes\OP(-2));\\
 \spl(E)(-1) &= (W^{\vee}\otimes W\otimes \Omega^1\otimes T(-1))\oplus (W^{\vee}\otimes U_2\otimes T(-1))\oplus (W^{\vee}\otimes U_3\otimes T(-2))
 \\&\oplus (W^{\vee}\otimes U_4\otimes T(-3))
  \oplus (U_2^{\vee}\otimes W\otimes\Omega^1(-1)) \oplus (U_2^{\vee}\otimes U_2\otimes\OP(-1))
  \\& \oplus (U_2^{\vee}\otimes U_3\otimes\OP(-2))
 \oplus (U_2^{\vee}\otimes U_4\otimes\OP(-3)) 
 \oplus  (U_3^{\vee}\otimes W\otimes\Omega^1) \oplus (U_3^{\vee}\otimes U_2\otimes\OP)
 \\&\oplus (U_3^{\vee}\otimes U_3\otimes\OP(-1)) 
 \oplus (U_3^{\vee}\otimes U_4\otimes\OP(-2)) 
  \oplus  (U_4^{\vee}\otimes W\otimes\Omega^1(1)) 
  \\&\oplus (U_4^{\vee}\otimes U_2\otimes\OP(1))
  \oplus (U_4^{\vee}\otimes U_3\otimes\OP)
 \oplus (U_4^{\vee}\otimes U_4\otimes\OP(-1));\\
 \Lambda^2 \E^{\vee}(-8) &= (S^2 W^{\vee} \otimes \Lambda^2T(-4)) \oplus (\Lambda^2 W^{\vee} \otimes S^2T(-4)) 
 \oplus (W^{\vee}\otimes U_2^{\vee}\otimes T(-4))
 \\&\oplus (W^{\vee}\otimes U_3^{\vee}\otimes T(-3)) 
 \oplus (W^{\vee}\otimes U_4^{\vee}\otimes T(-2)) 
 \oplus(\Lambda^2U_2^{\vee}\otimes\OP(-4)) 
 \\&\oplus(U_2^{\vee}\otimes U_3^{\vee}\otimes \OP(-3))
 \oplus(U_2^{\vee}\otimes U_4^{\vee}\otimes \OP(-2))
 \oplus(\Lambda^2U_3^{\vee}\otimes\OP(-2))
 \\&\oplus(U_3^{\vee}\otimes U_4^{\vee}\otimes \OP(-1))
 \oplus(\Lambda^2U_4^{\vee}\otimes\OP)
\end{align*}
Computing the cohomology groups of these bundles on $\p^3$, we get the result. 
\end{proof}

% ????????? 
% the block~$\Phi_{ij}$ is a component of $\Phi$ lying in the subspace $H^0(\p^3,\E_i\otimes\E_j(7))\subset H^0(\p^3,S^2\E(7))$ if~$i\ne j$,
% or in the subspace $H^0(\p^3,S^2\E_i\otimes\OP_{\p^3}(7))$, if~$i=j$. 

The decomposition \eqref{eq_e_i} induces the decomposition of the morphism $\phi$ from \eqref{eq_es_for_Iw}. 
We consider the bundles $S^2\E(6)$ and $\spl(E)(-1)$ as subbundles of $\E\otimes\E(6)$ and $\E\otimes\E^{\vee}(-1)$.
Thus, we introduce 
\begin{equation*}
 \phi_i^{jj'}\colon \E_i\otimes \E^{\vee}_j(-1) \to \E_i\otimes\E_{j'}(6),
\end{equation*}
such that the matrix $(\phi_i^{jj'})$ defines morphism $\phi$ in the exact sequence \eqref{eq_es_for_Iw}.
Note that each component $\phi_i^{jj'}$ is induced by the component $\Phi_{jj'}$ of $\Phi$.
This decomposition allows us to describe properties of the morphism between cohomologies of $S^2\E(6)$ and $\spl(E)(-1)$ induced by $\phi$:
\begin{equation*}
 D\colon H^1(\p^3,\spl(\E)(-1))\to H^1(\p^3,S^2\E(6)).
\end{equation*}
\begin{lemma}\label{lemma_k=1_or_2}
Let $D$ be the differential induced by $\phi$. 
\begin{enumerate}
 \item[$(1)$] If $k=1$, then $\dim(\Coker(D))\geqslant6+m_4-m_3$;
 \item[$(2)$] If $k=2$ and $m_3=4$, then $\Coker(D)\ne 0$.
 \item[$(3)$] If $k>0$ and $m_2=1$, then $\dim(\Coker(D))\geqslant 4$.
%  \item If $k>3$ and $m_2=1$ then $\Coker(D)$ is at least $3k(k+1)$-dimensional.
\end{enumerate}
\end{lemma}
\begin{proof}
 By Lemma \ref{lemma_cohomologies_F_i} and the decomposition of $\phi$ we see that $D$ is the following linear map:
\begin{equation*}
 D =  \begin{pmatrix} \phi^{11}_1 & \phi^{31}_1 \\ \phi^{14}_1 & \phi^{34}_1 \end{pmatrix} 
 \colon \left(W\otimes W^{\vee} \otimes V\right)\ \oplus\ \left(W\otimes U_3^{\vee}\right) \to \left(S^2W\otimes \Lambda^2 V\right) \oplus \left(W\otimes U_4\right).
\end{equation*}
Maps $\phi_1^{ij}$ are induced by the components $\Phi_{ij}$ of $\Phi$. By construction
\begin{equation}\label{eq_Phi_ij}
\begin{split}
\Phi_{14}\in  H^0(\p^3,\E_1\otimes\E_4(7))& = W\otimes U_4 \otimes H^0(\p^3,\Omega^1_{\p^3}(1))=0,\\
 \Phi_{11} \in H^0(\p^3,S^2\E_1\otimes\OP_{\p^3}(7)) &=H^0(\p^3, S^2(\Omega^1(-2)\otimes W)\otimes\OP_{\p^3}(7)) = \Lambda^2 W\otimes  V.
 \end{split}
\end{equation}
Thus the kernel of the restriction $D|_{W\otimes W^{\vee}\otimes V}$ coincides with the $\Ker(\phi_1^{11})$.
Moreover, if we fix $\Phi_{11} = \sum_{i=1}^l w'_i\wedge w''_i \otimes v_i$, then for any element $\xi\otimes v\in \Hom(W,W)\otimes V = W\otimes W^{\vee}\otimes V$
we have the following
\begin{equation}\label{eq_phi_111}
 \phi_1^{11}(\xi\otimes v) = \sum_{i=1}^l (\xi(w'_i) w''_i - \xi(w''_i)w'_i)\otimes (v\wedge v_i).
\end{equation}
If $k=1$ then by \eqref{eq_Phi_ij} we have that $D|_{W\otimes W^{\vee}\otimes V} = 0$, then
\begin{equation*}
 \dim(\Coker(D)) \geqslant \dim(H^1(\p^3,F_0)) - \dim(W\otimes U_3^{\vee}) = 6+m_4-m_3.
\end{equation*}
If $k = 2$ and $m_3 = 4$, then $\Phi_{11} = w'\wedge w'' \otimes v_1$ and by \eqref{eq_phi_111} 
for any $\xi\otimes v\in \Hom(W,W)\otimes V$ the image $\phi_1^{11}(\xi\otimes v)$ is a multiple  
of $v_1$. Thus, by \eqref{eq_Phi_ij} we get that $\dim(D(W\otimes W^{\vee})\otimes V)\leqslant 9$.
Then
\begin{equation*}
 \dim(D(H^1(\p^3,F_1))) \leqslant 9+ km_3 = 17 < 18  = \dim(S^2W\oplus \Lambda^2V) \leqslant \dim(H^0(\p^3,F_1)),
\end{equation*}
and $\Coker(D)$ is non-zero.

If $m_2 = 1$ and $k\ne 0$, then by \eqref{eq_phi_111} the element $\mathrm{Id}_W\otimes v \in \Hom(W,W)\otimes V$ is in $\Ker(D)$ for any $v\in V$.
Thus, $\Ker(D)$ is at least of dimension $4$. Moreover, in this case
\begin{equation*}
 \dim(H^1(\p^3,\spl(\E)(-1))) = \dim(H^1(\p^3,S^2\E(6))).
\end{equation*}
Thus, $\dim(\Ker(D) = \dim(\Coker(D))$ and we get the result.
\end{proof}

Now we are ready to prove Proposition \ref{prop_techlemma}.
\begin{proof}[Proof of Proposition \ref{prop_techlemma}.]
Consider the resolution $F_{\bullet} \to \I_w(5)$ in \eqref{eq_es_for_Iw} and the spectral sequence
$E^{pq}_1 = H^p(\p^3, F_q)\Rightarrow H^{p-q}(\p^3,\I_w(5))$:
\begin{equation*}
 \xymatrix{ \Lambda^2 U_2^{\vee} &0&0\\
 W^{\vee}\otimes U_2^{\vee}\ar@{-->}[rrd]^{\hspace{20pt}} & 0 & 0 \\
 S^2W^{\vee} \ar[r]^{\hspace{40pt}} \ar@{-->}[rrd] & \left( V\otimes  W\otimes W^{\vee}\right)\oplus \left(W\otimes U_3^{\vee}\right) \ar[r]^{D} &  
\left(\Lambda^2 V\otimes S^2 W\right) \oplus \left( W\otimes U_4\right)\\
 H^0(\p^3,F_2) \ar[r]& H^0(\p^3,  F_1) \ar[r] 
 & H^0(\p^3,F_0) 
 }
\end{equation*}
This spectral sequence converges to cohomologies of $\I_w(5)$. By assumption the defect of the even set of nodes $w$ vanishes;
thus by Lemma \ref{lemma_d=h^1} we get that 
\begin{equation}\label{eq_d=0_implies}
 E^{i,j}_{\infty}=0 \text{ for all } i\ne j.
\end{equation}
In particular, this implies that
$ 0 = E^{3,2}_{\infty} = E^{3,2}_1 = \Lambda^2 U_2$;
thus, $m_2$ equals either 0 or 1. We are going to show that in all cases except ones mentioned in Proposition \ref{prop_techlemma}
we get a contradiction with \eqref{eq_d=0_implies} or with assumption of the proposition.

First exclude the case $m_2=0$ and $k=1$. By the third assertion of Lemma \ref{conditions} and the first assertion of 
Lemma \ref{lemma_k=1_or_2} we have a contradiction with \eqref{eq_d=0_implies}:
\begin{equation*}
 \dim(E^{1,0}_{\infty})\geqslant \dim(\Coker(D)) - \dim(E^{2,2}_1) = 6+m_4-m_3-m_2 = k= 1.
\end{equation*}

If $m_2=0$ and $k=2$, then by the second and the third assertions of Lemma \ref{conditions} we have that~\mbox{$m_3 = 4$}.
By Lemma \ref{lemma_k=1_or_2} we get that $\dim(\Coker(D))\ne 0$, then we have a contradiction with \eqref{eq_d=0_implies}:
\begin{equation*}
 \dim(E^{1,0}_{\infty})\geqslant \dim(\Coker(D)) - \dim(E^{2,2}_1) > - km_2 = 0.
\end{equation*}
Thus, we showed that if $m_2=d(w)=0$, then either $k=0$ or $k>3$. 

If $m_2 = 1$ and $0<k<4$, then by third assertion of Lemma \ref{lemma_k=1_or_2} we get~\mbox{$\dim(\Coker(D))\geqslant 4$.}
This implies the contradiction to \eqref{eq_d=0_implies}:
\begin{equation*}
  0=\dim(E^{1,0}_{\infty})\geqslant \dim(\Coker(D)) - \dim(E^{2,2}_1) \geqslant 4-k >0.
\end{equation*}
Finally, if $m_2=1$ and $k\geqslant 4$, then by the third assumption of Lemma \ref{conditions} we have
\begin{equation*}
 m_4 = m_3+k-3 > m_3.
\end{equation*}
In view of Lemma \ref{lemma_cond_on_m_4}, this contradicts the irreducibility of $B$,
so this case is also excluded.

Thus, if $m_2=1$ and $d(w) = 0$, then $k=0$. This finishes the prove. 
\end{proof}

Now let us provide here some useful computations:
\begin{lemma}\label{lemma_computation_of_d}
Assume that $w$ is a $\delta/2$-even set of nodes on a nodal sextic surface $B$, and~$\E$ is its Casnati-Catanese bundle. 
\begin{enumerate}
 \item If $\delta = 0$ and  $\E= \OP_{\p^3}(-2)^{\oplus 3}$ or $\E =\Omega_{\p^3}^1(-1)\oplus \OP_{\p^3}(-2)$, then $d(w) = 0$;
 \item If $\delta = 1$ and $\E = \OP_{\p^3}(-3)^{\oplus 3}\oplus\OP_{\p^3}(-2)$ or $\E =  \OP_{\p^3}(-3)^{\oplus 6}$, then $d(w) = 0$.
\end{enumerate}
\end{lemma}
\begin{proof}
 Let us start with a case $ \E = \OP_{\p^3}(-2)^{\oplus 3}$. Substituting it into 
the exact sequence \eqref{eq_es_for_Iw}, we get the following:
\begin{equation*}
  0 \to \OP_{\p^3}(-3)^{\oplus 3}\to \OP_{\p^3}(-1)^{\oplus 8} \to \OP_{\p^3}(1)^{\oplus 6} \to \I_w(5) \to 0.
\end{equation*}
Therefore, we get $h^1(\p^3,\I_w(5))= d(w)=0$ by Lemma \ref{lemma_d=h^1}.

If $\E=\Omega_{\p^3}^1(-1)\oplus \OP_{\p^3}(-2)$, then the resolution \eqref{eq_es_for_Iw} of $\I_w(5)$ is as follows:
\begin{multline*}
 0\to \Lambda^2T_{\p^3}(-5) \oplus T_{\p^3}(-4) \to T_{\p^3}\otimes\Omega_{\p^3}^1(-1)\oplus T_{\p^3}(-2) \oplus \Omega_{\p^3}^1 \to \\ \to
 S^2\Omega_{\p^3}^1(3) \oplus \Omega_{\p^3}^1(2) \oplus \OP_{\p^3}(1) \to \I_w(5) \to 0.
\end{multline*}
Then in view of the following table of cohomologies we get that $h^1(\p^3,\I_w(5))=d(w)=0$ by Lemma \ref{lemma_d=h^1}.
 \begin{center}
 \begin{tabular}{|c|c|c|c|c|}
 \hline
 $F$ & $H^0(\p^3,F)$ & $H^1(\p^3,F)$ & $H^2(\p^3,F)$ & $H^3(\p^3,F)$ \\
 \hline
 \hline
 $\Lambda^2 T_{\p^3}(-5)$ & $0$ &$0$ &$\det(V)$ &$0$ \\
 \hline
 $T_{\p^3}(-4)$ & $0$ &$0$ &$\det(V) 	$ &$0$ \\
 \hline
 \hline
 $T_{\p^3}\otimes\Omega_{\p^3}^1(-1)$ & $0$ &$V$ &$0$ &$0$ \\
  \hline
 $T_{\p^3}(-2)$ & $0$ &$0$ &$0 	$ &$0$ \\
 \hline
 $\Omega_{\p^3}^1$ & $0$ &$\mathbb{C}$ &$0$ &$0$ \\
 \hline
 \hline
 $S^2\Omega_{\p^3}^1(3)$  & $0$ & $0$ &$0$ &$0$ \\
 \hline
 $\Omega_{\p^3}(2)$ & $\Lambda^2V^{\vee}$ &$0$ &$0$ &$0$ \\
 \hline
 $\OP_{\p^3}(1)$ & $V^{\vee}$ &$0$ &$0$ &$0$ \\
 \hline
\end{tabular}
\end{center}

Finally, consider bundles $\E = \OP_{\p^3}(-3)^{\oplus 3}\oplus\OP_{\p^3}(-2)$ or $ \OP_{\p^3}(-3)^{\oplus 6}$. 
By Lemma~\ref{lemma_cohomologies_F_i} we see that $h^1(\p^3,\I_w(5))=0$ in both cases. So $d(w)=0$ by Lemma~\ref{lemma_d=h^1}.
\end{proof}

\section{Proof of Theorem \ref{Result}}\label{sect_proof}

In this section we finally prove Theorem~\ref{Result} and specify surfaces  such that double covers 
of $\p^3$ branched along them have a non-trivial group $T_2(\resx)$. 
First, we list all sextic surfaces containing a minimal $1/2$-even set of nodes with zero defect.

\begin{proposition}\label{prop_1_2_even_d=0}
 If $B$ is an irreducible nodal sextic surface, $w$ is a minimal $1/2$-even set of nodes on $B$  with $d(w)=0$,
 then in the notation of Proposition $\ref{description_of_E_1_2_even}$ either $(k,m_2,m_3,m_4)$ is equals  to $(0,0,6,0)$ or to $(0,1,3,0)$. 
\end{proposition}
\begin{proof}
Since $d(w)=0$ by Proposition \ref{prop_es_cut_by_plane} the even set of nodes $w$ is not cut out by a plane. 
Proposition \ref{prop_techlemma} implies that either $m_2=0$ or $1$.

We start with the case $m_2=0$.  
By the second and the third assertions of Lemma \ref{conditions} we have that~\mbox{$m_4=0$}.
Thus, $H^0(\p^3,\F(3))$ generates the graded $R$-module $M$ defined in \eqref{M}. In particular, we have a surjective morphism
\begin{equation*}
 \varepsilon\colon H^0(\p^3,\F(3))\otimes V^{\vee} \twoheadrightarrow H^0(\p^3,\F(4)).
\end{equation*}
This implies that $4h^0(\p^3,\F(3))\geqslant h^0(\p^3,\F(4))$. Then by Lemmas \ref{conditions} and \ref{lemma_RR} we have that 
\begin{equation*}
  4m_3  = 4h^0(\p^3,\F(3))  \geqslant h^0(\p^3,\F(4)) = h^0(\p^3,\F(3))+12 = m_3+12.
\end{equation*}
Thus, $m_3 \geqslant 4$; then, by the second assertion of Lemma \ref{conditions} we get that $k\leqslant2$.
Then Proposition \ref{prop_techlemma} implies that~\mbox{$k=0$}. Thus, the only  possible collection $(k,m_2,m_3,m_4)$ is~\mbox{$(0,0,6,0)$}. By 
Lemma \ref{lemma_computation_of_d} in this case $d(w)$ is zero.

Now assume $m_2=1$. By Proposition \ref{prop_techlemma} we get that $k=0$. 
Consider the Koszul exact sequence tensored by $\F$ (see Notation \ref{def_of_F}). It is the following exact sequence on $\p^3$:
\begin{equation*}
 0\to \F \to \F(1)\otimes V \to \F(2)\otimes \Lambda^2V^{\vee} \to \F(3)\otimes V^{\vee} \to \F(4) \to 0. 
\end{equation*}
The spectral sequence of cohomologies~\mbox{$E_1^{p,q} = H^p(\p^3,\F(q))\otimes \Lambda^q V$} is as follows:
\begin{equation*}
 \xymatrix{
 H^2(\F) \ar[r]^{\rho\ \ \ \ } \ar@{-->}[ddrrr]^{D'} &  H^2(\F(1))\otimes V & 0 & 0& 0\\
 0 & 0 & 0& 0& 0 \\
  0 & 0 & H^0(\F(2))\otimes \Lambda^2 V^{\vee}\ar[r]& H^0(\F(3))\otimes V^{\vee} \ar[r]^{\ \ \ \varepsilon} &  H^0(\F(4)) 
 }
\end{equation*}
By \eqref{eq_Serre_duality} the map $\rho$ is the Serre dual to the canonical morphism $\gamma$ in Lemma \ref{lemma_m_2=1_then_inclusion}. 
Since~$\gamma$ is an embedding,~$\rho$ is surjective and $E_2^{2,1}  = \Coker(\rho)= 0$. Then $D'$ is an only non-zero
differential on pages $E_i$ for $i>1$.
We started with the exact sequence; thus, the spectral sequence converges to zero. In particular, we get the following equality:

\begin{equation*}
 0=E_{\infty}^{0,4} = E_2^{0,4}=\mathrm{Coker}(\varepsilon). 
\end{equation*}
Thus, the graded~$R$-module $M$ is generated by components $H^0(\p^3,\F(2))$ and $H^0(\p^3,\F(3))$. 
This implies $m_4 = 0$. Then, by the third assertion of Lemma \ref{conditions}, we have $m_3=3$.
Thus, it is the case  $(k,m_2,m_3,m_4)=(0,1,3,0)$ and by Lemma \ref{lemma_computation_of_d} the defect $d(w)$ vanishes.
\end{proof}

Now we specify all $0$-even sets of nodes with zero defect.
\begin{proposition}\label{description_0_of_d(w)=0}
 Assume that $B$ is a nodal sextic surface and $w$ is a $0$-even set of nodes on $B$ such that~\mbox{$d(w) = 0$}. 
 Then~\mbox{$|w|=32$} or~$40$, and the Casnati--Catanese bundle of $w$ is isomorphic to $\OP_{\p^3}(-2)^{\oplus 3}$ or~\mbox{$\Omega_{\p^3}^1(-1)\oplus \OP_{\p^3}(-2)$}
 respectively.
\end{proposition}
\begin{proof}
  By Theorem~\ref{0-even_sets} we have an explicit description of 
 $0$-even sets of nodes on nodal sextic surfaces. In this case there appear only four different 
 Casnati--Catanese bundles, which are listed in Theorem~\ref{0-even_sets}.
 
 If $|w| = 56$, then $d(w) \ne 0$, since $h^0(\OP_{\p^3}(5)) = 56$ 
 and $w$ is contained in the zero loci of the partial derivatives of the equation of $B$, which are quintics. 
 
 If $|w| = 24$, then $d(w)\ne 0$ by \cite[Lemma 3.1]{Endrass99}.
 
 If $|w| = 32$ or 40, then by Lemma \ref{lemma_computation_of_d} in both case we get that~\mbox{$d(w)=0$.}
\end{proof} 

Now we are ready to describe nodal sextic surfaces providing obstructions to rationality of double solids.
\begin{proof}[Proof of Theorem $\ref{Result}$]
 Assume that $B$ is a nodal sextic surface and $T_2(\resx)\ne 0$. By~Theorem~\ref{Clemens} it implies 
 that $\dim(\codes)>d(\Sing(B))$. Then by Corollary~\ref{codes_and_defect}
 there exists a non-trivial even set of nodes~\mbox{$w\in \codes$} such that $d(w) = 0$. 
 
 If $w$ is a $0$-even set of nodes, then by Proposition~\ref{description_0_of_d(w)=0} its Casnati--Catanese
 bundle is isomorphic to either~\mbox{$\OP_{\p^3}(-2)^{\oplus 3}$} or $\Omega_{\p^3}^1(-1)\oplus \OP_{\p^3}(-2)$.

 If $w$ is a $1/2$-even set of nodes, then we may assume that it is minimal. Otherwise, we can 
 replace it by its proper $1/2$-even subset and its defect should also vanish by Lemma~\ref{defect_of_subsets}.
%  Moreover, by Proposition \ref{prop_es_cut_by_plane} we get that $w$ is not cut out by a plane.
 Then by Proposition~\ref{prop_1_2_even_d=0} we get the result. 
\end{proof}

Theorem~\ref{Result} gives us an explicit description of surfaces which can arise as branch surfaces of sextic
double solids with non-vanishing Artin--Mumford obstruction to rationality. However, we can not ensure that any surface 
of this type actually admits this obstruction. Nevertheless, Proposition \ref{examples} claims that a general surface of 
this type indeed admits it.

\begin{proof}[Proof of Proposition~\ref{examples}]
 Any vector bundle $\E$ from the list in Theorem~\ref{Result} except the last one is a sum of line bundles over $\p^3$. Thus,~\mbox{$S^2\E(6+\delta)$} 
 is also a sum of line bundles, and all these bundles are of positive degrees. Therefore, the 
 bundle~\mbox{$S^2\E(6+\delta)$} is generated by its global sections.
 
 If~\mbox{$\E = \Omega_{\p^3}^1(-1)\oplus \OP_{\p^3}(-2)$},
 then $S^2\E(6) = S^2 \Omega_{\p^3}^1(4)\oplus\Omega_{\p^3}^1(3)\oplus \OP_{\p^3}(2)$. Since each summund is generated
 by its global sections, one has the same for the bundle $S^2\E(6)$.
 
 Therefore, by~Theorem~\ref{Catanese_Casnati} singularities of a general surface $B$ 
 in each of these families consist just of the minimal even set of nodes~\mbox{$w(\E)$}.
 The defect of this even set of nodes vanishes by Lemma~\ref{lemma_computation_of_d}, so
 by formula~\eqref{decomp-codes} we get that $T_2(\resx)\ne 0$.
\end{proof}

\section{Discussion}\label{sect_conclusion}
\subsection*{Examples with positive defect}
By Theorem~\ref{Result} we showed that all sextic double solids admitting
a non-trivial Artin--Mumford obstruction to rationality are necessarily branched in symmetric surfaces of
four exactly defined types. Moreover, by Proposition~\ref{examples} we get that general surfaces of this types
gives actual examples of double solids with a non-trivial Artin--Mumford obstruction. Remarkably, for any general
surface~$B$ of each type the space $\codes$ is one-dimensional by~Theorem~\ref{Catanese_Casnati}. So we do not know 
any examples of a surface $B$ such that $\dim(\codes)>1$ and $T_2(\resx) \ne 0$.

However, in~\mbox{\cite[Proposition 3.3]{CatTon}} is given a bundle $\E$ which defines a surface $B$ with an even set~$w$ of~56 
nodes which consists of all singularities of $B$. Another construction of this variety is given in
\cite{Geemen-Zhao}.
By~\mbox{\cite[Theorem 4.5]{Jaffe-Ruberman}} in the notation of
Section~\mbox{\ref{sect_minimal_1/2_even}} we have $\dim(\codes) = 1 + 2h^1(\p^3,\F)$. However, by the 
description~\mbox{\cite[Proposition 3.3]{CatTon}}, we can conclude that 
\begin{equation*}
 \dim(\codes) = 7.
\end{equation*}
Also the defect $d(\Sing(B))>0$ in this case, because all $56$ points of $\Sing(B)$ are contained in 
four quintic surfaces which are zero loci of partial derivative of the equation of $B$. So if
$d(\Sing(B))<7$ we would get the new family of examples of non-rational sextic double solid and this 
non-rationality would not be a corollary of~\mbox{\cite[Corollary A]{Cheltsov_Park}}.

The group $\codes$ is generated by some set of minimal even sets of nodes, moreover, if $w'\in \codes$, 
then we have~\mbox{$w\setminus w'\in \codes$}. So $\codes$ is generated by even sets of $15,\ 24,\ 32$ and $39$ nodes.
Moreover, since the sum of any two even sets of nodes is also even, we can conclude, that $\codes$ 
is generated just by even sets of $24$ and~$32$ nodes.

Then we have at least two possible ways to estimate $d(\Sing(B))$. The first one is to make a straightforward 
computation similar to the one in Section~\mbox{\ref{sect_computations_of_defect}} using the vector bundle $\E$ constructed
\mbox{in \cite[Proposition 3.3]{CatTon}}. The second way to estimate the defect is to use the structure of the 
group~$\codes$ and to connect somehow the defect $d(\Sing(B))$  with defects of all $w\ne w'\in \codes$ which 
defects are less than or equal to 1. However, we cannot find any reasonable bound even for the defect of the disjoint 
union of two sets of points with known defect. The only very rough idea is that it is less than the cardinality 
of the smaller set, but this is insufficient for our goals. 

\subsection*{Very general sextic double solid}

By the degeneration method introduced in \cite{Voisin}, the existence of sextic double solids with Artin--Mumford 
obstructions to stable rationality implies that a very general nodal sextic double solid from a family containing
such an example is not stably rational. Therefore, it would be interesting to find out have the examples provided 
by Theorem \ref{Result} fit into the parameter space of nodal sextic double solids, similarly to what was done
in \cite[Section 2]{Voisin}. Note that the existence of any of these examples implies that a very general smooth sextic
double solid is not stably rational, but this also follows from the more general result of \cite{HT}.

\subsection*{Other double covers}

It would be interesting to obtain a classification similar to \cite{Endrass99} and Theorem \ref{Result} for other Fano
threefolds that have a double cover structure. It seems that except quartic and sextic double solids three-dimensional 
Artin--Mumford-type examples are known only for a particular family of double covers of quartics over an intersection
with a quartics which have 20 nodes, see \cite[Theorem 1.2]{PS}.
In particular, it would be interesting to know if such examples exist for 
Fano threefolds of type $X_{10}$ which are double covers of the del Pezzo threefold $V_5$, for del Pezzo threefolds 
of type $V_1$, for complete intersection of quadrics and cubics~$X_6$ that are double covers of cubic hypersurfaces and for complete intersection of three 
quadrics~$X_8$  that are double covers of intersection of two quartics.

\subsection*{Casnati--Catanese bundles}

Theorem \ref{if_codes_ne_0} gives necessary conditions for Casnati--Cata\-ne\-se bundles of minimal $1/2$-even sets
of nodes. It would be interesting to find a complete classification of such Casnati--Catanese bundles, and furthermore, of
Casnati--Catanese bundles of arbitrary 
$1/2$-even sets of nodes
similar to that of Theorem~\ref{0-even_sets}.

\bibliographystyle{alpha}
\bibliography{even_sets}

\end{document}